\documentclass[11pt,english]{article}
\usepackage[T1]{fontenc}
\usepackage[latin9]{inputenc}
\usepackage{geometry}
\usepackage{calc}
\usepackage{units}
\usepackage{mathrsfs}
\usepackage{amsthm}
\usepackage{amsmath}
\usepackage{amssymb}
\usepackage{stmaryrd}
\usepackage{graphicx}
\usepackage{esint}

\makeatletter
\numberwithin{equation}{section}
\numberwithin{figure}{section}
\numberwithin{table}{section}
\newenvironment{lyxcode}
{\par\begin{list}{}{
\setlength{\rightmargin}{\leftmargin}
\setlength{\listparindent}{0pt}
\raggedright
\setlength{\itemsep}{0pt}
\setlength{\parsep}{0pt}
\normalfont\ttfamily}%
 \item[]}
{\end{list}}
  \theoremstyle{plain}
  \newtheorem*{thm*}{\protect\theoremname}
\theoremstyle{plain}
\newtheorem{thm}{\protect\theoremname}[section]
  \theoremstyle{plain}
  \newtheorem{cor}[thm]{\protect\corollaryname}
  \theoremstyle{remark}
  \newtheorem*{acknowledgement*}{\protect\acknowledgementname}
  \theoremstyle{definition}
  \newtheorem{defn}[thm]{\protect\definitionname}
  \theoremstyle{remark}
  \newtheorem{claim}[thm]{\protect\claimname}
  \theoremstyle{definition}
  \newtheorem{example}[thm]{\protect\examplename}
  \theoremstyle{plain}
  \newtheorem{lem}[thm]{\protect\lemmaname}
  \theoremstyle{remark}
  \newtheorem{rem}[thm]{\protect\remarkname}
  \theoremstyle{plain}
  \newtheorem{conjecture}[thm]{\protect\conjecturename}
  \theoremstyle{plain}
  \newtheorem{prop}[thm]{\protect\propositionname}
\newcommand{\lyxaddress}[1]{
\par {\raggedright #1
\vspace{1.4em}
\noindent\par}
}

\usepackage{bbm}
\newcommand{\ind}{{\mathbbm{1}}}
\date{}

\makeatother

\usepackage{babel}
  \providecommand{\acknowledgementname}{Acknowledgement}
  \providecommand{\claimname}{Claim}
  \providecommand{\conjecturename}{Conjecture}
  \providecommand{\corollaryname}{Corollary}
  \providecommand{\definitionname}{Definition}
  \providecommand{\examplename}{Example}
  \providecommand{\lemmaname}{Lemma}
  \providecommand{\propositionname}{Proposition}
  \providecommand{\remarkname}{Remark}
  \providecommand{\theoremname}{Theorem}
\providecommand{\theoremname}{Theorem}

\begin{document}

\title{\textbf{Simplicial branching random walks}\\
\textbf{and their applications}}

\author{Ron Rosenthal%
\thanks{Partially supported by an ETH fellowship.%
}}
\maketitle
\begin{lyxcode}
\end{lyxcode}
\begin{abstract}
We define a new stochastic process on general simplicial complexes
which allows to study their spectral and homological properties. Some
results for random walks on graphs are shown to hold in this general
setting. As an application, the process is used to calculate the spectral
measure of high-dimensional analogues of regular trees and to construct
solutions to the high-dimensional Dirichlet problem for forms. 
\end{abstract}

\section{Introduction}

The topic of ``random walks on graphs'' is a classical and fundamental
subject. With a history of more than a century and a variety of applications
to physics, computer science, chemistry and many other fields, random
walks are among the most valuable stochastic models. In addition,
their connections with many areas of research within mathematics such
as probability, geometry, graph theory, harmonic analysis, group theory,
etc, make random walks a valuable tool when investigating their interplay.
Accordingly, the literature is very vast and we refer the reader to
the following books \cite{Sp76,DS84,Lo96,Wo00,LL10} as well as the
references therein for background on the subject. 

Simplicial complexes are combinatorial and topological extensions
of graphs and it is thus natural to ask whether one can generalize
random walk models to the world of high-dimensional simplicial complexes. 

A first construction of such a stochastic process was suggested in
\cite{PR12} by Parzanchevski and the author. The process, which is
called the $\left(d-1\right)$-random walk, reflects in its asymptotic
behavior spectral properties of the upper Laplacian (originating in
the work of Eckmann \cite{Eck44}) as well as homological properties
of the complex. In a subsequent work \cite{MS13}, Mukherjee and Steenbergen
constructed a similar model for random walks on simplicial complexes,
which is connected to the lower Laplacian and in particular allows
to study the top homology of the complex. 

The connection of both models to the spectral and homological properties
of simplicial complexes is done via the study of an associated ``process'',
called the expectation process, which takes the role played by the
heat kernel of a random walk in the graph case. This generalization
differs from ``classical'' heat kernels in two regards: first, due
to the fact that high-dimensional simplexes have two possible orientations,
it is defined as the difference of two probabilities. Secondly, in
order to extract information from this difference of probabilities,
which always converges to zero, a suitable normalization is required. 

In this paper, we present a new stochastic process, called \emph{simplicial
branching random walk}, which is connected to spectral and homological
properties of simplicial complexes in a similar way as the other processes.
However, unlike in the previous models, the fact that one needs to
look at the difference of two quantities can already be observed at
the level of the process \emph{itself} and not merely in the context
of an associated process. Moreover in this new process, there is no
need for normalization. Hence, we can work with it \emph{directly},
and are able to gain new insights regarding the nature of this process
as well as its connections to spectral and homological properties
of the complex. 

In recent years, there has been considerable interest in \emph{high-dimensional
expanders}, namely, analogues of expander graphs in the context of
general simplicial complexes. As in the graph case, stochastic processes
like the simplicial branching random walk and the $\left(d-1\right)$-random
walks constructed in \cite{PR12,MS13} are closely related to one
such notion of expansion, namely, spectral expansion. Other notions
of expansion include: combinatorial expansion \cite{PRT12,Par13,GS14},
geometric and topological expansion \cite{Gr10,FGLNP10,MW11}, $\mathbb{F}_{2}$-coboundary
expansion \cite{DK12,SKM12} and Ramanujan complexes \cite{CSZ03,Li04,LSV05,GP14,EGL14,KKL14}.
There is also a great interest in the behavior of random complexes.
The standard model for such complexes is the Linial-Meshulam model,
defined in \cite{LM06}, which has been extensively studied, see \cite{MW09,Ko10,Wa11,HKP12,HJ13Th,HKP13,LP14,LPa14};
see also \cite{LM13} for related results on a different model.

\subsection{The model}

Let us start with an informal description of the the model. A more
precise definition is postponed to Section \ref{sec:The--branching-random-walk}
after all required notation and terminology are introduced in Section
\ref{sec:Notation-and-some_useful_facts}. Let $X$ be a $d$-dimensional
complex. The \emph{simplicial branching random walk} (SBRW for short)
$\left(N_{n}\right)_{n\geq0}$ is a particle process on the set of
oriented $\left(d-1\right)$-simplexes of $X$, denoted by $X_{\pm}^{d-1}$,
where, for an oriented $\left(d-1\right)$-simplex $\sigma$ and $n\geq0$,
$N_{n}\left(\sigma\right)$ stands for the number of particles in
$\sigma$ at time $n$. The process $\left(N_{n}\right)_{n\geq0}$
is a time-homogeneous Markov chain on $\mathbb{N}^{X_{\pm}^{d-1}}$,
with transition kernel which is described by the following law: Given
a configuration of particles, each of the particles (simultaneously
and independently) chooses one of the $d$-simplexes containing the
$\left(d-1\right)$-simplex of its current position uniformly at random
and splits into $d$ new particles that are now located on the other
$d$ faces of the chosen $d$-simplex (with an appropriate choice
of orientation, see Section \ref{sec:The--branching-random-walk}). 

For example, if $X$ is a triangle complex, the SBRW is a particle
process on oriented edges. If a particle is positioned on the oriented
edge $\left[u,v\right]$ and the chosen triangle containing it is
$\left\{ u,v,w\right\} $ then the particle splits into two new particles
on $\left[u,w\right]$ and $\left[w,v\right]$ (the orientation is
chosen so that the original oriented edge and the new oriented edges
have the same origin or the same terminus). 

Given $0\leq p\leq1$, one can also discuss the $p$-lazy version
of the SBRW in which every particle stays put with probability $p$
and with probability $\left(1-p\right)$ acts according to the law
described above. An illustration of one step of the process for a
triangle complex can be found in Figure \ref{fig:One_step_of_the_branching_process}.

\begin{figure}[h]
\centering{}\includegraphics[scale=1.5]{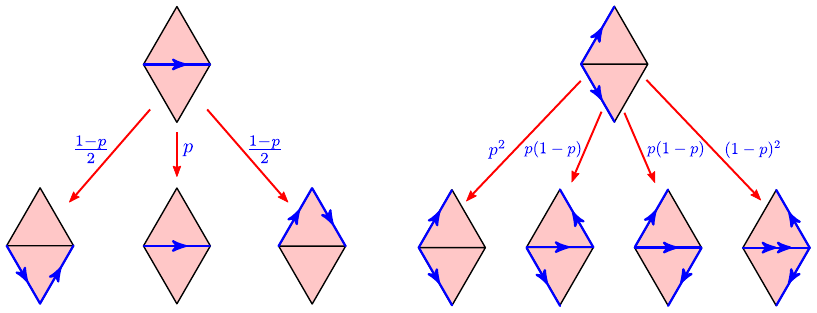}\protect\caption{One step of the simplicial branching random walk for two configurations.
On the left: The particle starting at the center stays put with probability
$p$ and with probability $1-p$ chooses one of the triangles containing
it uniformly at random and splits into two particles on the two other
(neighboring) edges of this triangle. On the right: each of the particles
stays put with probability $p$ or splits into two particles on the
unique triangle containing its current edge with probability $1-p$.
\label{fig:One_step_of_the_branching_process}}
\end{figure}

We now introduce the effective version of the process called\emph{
effective simplicial branching random walk} \emph{(ESBRW for short)}
by 
\[
D_{n}\left(\sigma\right)=N_{n}\left(\sigma\right)-N_{n}\left(\overline{\sigma}\right),
\]
where for an oriented $\left(d-1\right)$-simplex $\sigma$, $\overline{\sigma}$
is the same $\left(d-1\right)$-simplex with the opposite orientation.
Finally the heat kernel is defined as 
\[
\mathscr{E}_{n}\left(\sigma,\sigma'\right)=E^{\sigma}\left[D_{n}\left(\sigma'\right)\right],
\]
where $E^{\sigma}$ denotes the expectation when starting with a unique
particle on the oriented $\left(d-1\right)$-simplex $\sigma$.

\subsection{Main results }

We now give a description of the main results. For the sake of clarity
we only give informal statements for some of the results and refer
the reader to later sections for the precise statements.

The first result deals with the connection between the asymptotic
behavior of the heat kernel and the existence of non-trivial homology
in finite complexes. It is shown that a similar relation to the one
proved in \cite{PR12} for the $\left(d-1\right)$-walk holds for
the ESBRW (see Theorem \ref{thm:Finite_complexes_BRW_vs_homology}
for the precise statement). 
\begin{thm*}
Let $X$ be a \textbf{finite $d$}-complex, $\left(D_{n}\right)_{n\geq0}$
the $p$-lazy ESBRW on $X$ and $\left(\mathscr{E}_{n}\right)_{n\geq0}$
its heat kernel. If $p>\frac{d-1}{d+1}$, then 
\begin{enumerate}
\item The limit $\mathscr{E}_{\infty}=\lim_{n\to\infty}\mathscr{E}_{n}$
always exists. 
\item One can read off from the family $\left\{ \mathscr{E}_{\infty}\left(\sigma,\cdot\right)\right\} _{\sigma\in X_{\pm}^{d-1}}$
the dimension of the $\left(d-1\right)$-homology, and in particular,
whether the homology is trivial.
\item If furthermore $p\geq\frac{d}{d+1}$, then the rate of convergence
of $\mathscr{E}_{n}$ is exponential with a constant that depends
on a high-dimensional analogue of the spectral gap. 
\end{enumerate}
\end{thm*}
Our next result concerns a generalization of the following important
identity for random walks on graphs (see also Theorem \ref{thm:R_via_S_and_overline=00007BS=00007D}). 
\begin{thm*}
Let $P^{v}$ $(E^{v}$) be the law (expectation) of a random walk
on a graph $G$. The identity 
\begin{equation}
E^{v}\left[\begin{array}{c}
\mbox{number of visits to }v\\
\mbox{by the random walk}
\end{array}\right]=\frac{1}{1-P^{v}\left(\begin{array}{c}
\mbox{The random walk}\\
\mbox{returns to }v
\end{array}\right)}\label{eq:return_prob_RW}
\end{equation}
has a high-dimensional analogue for the effective simplicial branching
random walk.
\end{thm*}
Next, we discuss some applications of ESBRW to the study of simplicial
complexes. The $d$-dimensional counterpart of the $k$-regular tree,
called $k$-regular arboreal $d$-complex, was defined in \cite{PR12}.
It is obtained by attaching to a $\left(d-1\right)$-simplex $k$
new $d$-simplexes and then adding recursively to every new $\left(d-1\right)$-simplex
$\left(k-1\right)$ new $d$-simplexes (see also Definition \ref{def:regular_arboreal_complex}).
By generalizing ideas of Kesten \cite{Ke59} to ESBRW, we are able
to find the spectral measure of the ``transition'' operator $\mathscr{A}_{0}$
(see Lemma \ref{lem:comparison_of_models} for the definition). 
\begin{thm}
\label{thm:The_spectral_measure_of_arboreal_complexes}The spectral
measure $\mu_{d,k}$ of $\mathscr{A}_{0}=I-\Delta^{+}$ for the $k$-regular
arboreal $d$-complex is given by 
\[
\mu_{d,k}\left(A\right)=\begin{cases}
\int_{A}\rho_{d,k}\left(x\right)dx+\frac{d+1-k}{d+1}\chi_{1\in A} & ,\quad k<d+1\\
\int_{A}\rho_{d,k}\left(x\right)dx & ,\quad k\geq d+1
\end{cases},
\]
where $\chi$ is the indicator function, 
\[
\rho_{d,k}\left(x\right)=\frac{\sqrt{4\left(k-1\right)d-\left(kx+\left(d-1\right)\right)^{2}}}{2\pi\left(d+x\right)\left(1-x\right)}\chi_{x\in I_{d,k}}
\]
and 
\[
I_{d,k}=\left[\frac{1-d-2\sqrt{\left(k-1\right)d}}{k},\frac{1-d+2\sqrt{\left(k-1\right)d}}{k}\right].
\]

\end{thm}
In particular, this gives a new proof of the fact that the spectrum
of $\mathscr{A}_{0}$ is $I_{d,k}$ for $k\geq d+1$ and $I_{d,k}\cup\left\{ 1\right\} $
when $k<d+1$, which is the content of \cite[Theorem 3.3]{PR12}.

As a corollary of Theorem \ref{thm:The_spectral_measure_of_arboreal_complexes}
we obtain the following transience/recurrence classification for regular
arboreal complexes:
\begin{cor}
\label{cor:Transience_and_recurrence_of_T^d_k} The effective simplicial
branching random walk on $T_{k}^{d}$ is recurrent, i.e., $\sum_{n=0}^{\infty}\mathscr{E}_{n}^{p}\left(\sigma,\sigma\right)=\infty$
for every $p>\frac{d-1}{d+1},$ if $k\leq d+1$, and transient if
$k>d+1$. 
\end{cor}
Note that this implies the same recurrence/transience classification
for the $\left(d-1\right)$-random walk from \cite{PR12}.

Our last result concerns the Dirichlet problem on simplicial complexes.
Recall that for a finite graph $G=\left(V,E\right)$, $\emptyset\neq A\subset V$
and $f:A\to\mathbb{R}$ the unique solution to the Dirichlet problem%
\footnote{The Dirichlet problem for a given triplet $\left(G,A,f\right)$ is
to find  a solution $F:V\to\mathbb{R}$ to the boundary value problem
$\Delta F=0$ on $V\backslash A$ and $F=f$ on $A$, where $\Delta$
is the graph Laplacian.%
}, can be written using the random walk as $F\left(v\right)=E^{v}\left[f\left(Y_{\tau}\right)\right]$,
where $\left(Y_{n}\right)_{n\geq0}$ is the simple random walk on
$G$ and $\tau=\inf\left\{ k\geq0\,:\, Y_{k}\in A\right\} $. 

In Section \ref{sec:Dirichlet-problem}, we discuss the high-dimensional
analogue of the Dirichlet problem and show the following:
\begin{thm*}
For every finite complex $X$, every subset $A$ of the $\left(d-1\right)$-simplexes
satisfying a certain homological condition and every form $f$ on
$A$, there exists a unique solution to the Dirichlet problem that
can be expressed in terms of the ESBRW. 
\end{thm*}

\subsection{Structure of the paper}

The remainder of this paper is organized as follows: 

In Section \ref{sec:Notation-and-some_useful_facts} we introduce
the relevant notation and definitions regarding: simplicial complexes
(Subsection \ref{subsec:Simplicial-complexes}), high-dimensional
Laplacians (Subsection \ref{subsec:High-dimensional-Laplacians}),
discrete Hodge theory (Subsection \ref{subsec:Discrete-Hodge-theory})
and the $\left(d-1\right)$-walk (Subsection \ref{subsec:The--walk}). 

In Section \ref{sec:The--branching-random-walk}, we define the SBRW
and the ESBRW, discuss some of their basic properties and prove the
first two main results, Theorem \ref{thm:Finite_complexes_BRW_vs_homology}
and Theorem \ref{thm:R_via_S_and_overline=00007BS=00007D}. 

Section \ref{sec:arboreal_complexes} deals with application of the
ESBRW to the study of arboreal complexes and provides the proof of
Theorem \ref{thm:The_spectral_measure_of_arboreal_complexes} and
Corollary \ref{cor:Transience_and_recurrence_of_T^d_k}. 

In Section \ref{sec:Dirichlet-problem}, the high-dimensional Dirichlet
problem is discussed, and in particular how the ESBRW can be used
to construct its solutions. 

Section \ref{sec:Lower---branching_random_walk} explains how to construct
a similar particle process corresponding to the lower Laplacian, thus
allowing to generalize most of the results from previous sections
to this setting. 

The appendix provides the proof of some claims stated throughout the
manuscript. 

\bigskip{}

\begin{acknowledgement*}
This research has been partially supported by an ETH fellowship. The
author would like to thank Mayra Bermúdez, Adrien Kassel, Xinyi Li
and Pierre-François Rodriguez for some useful discussions. 
\end{acknowledgement*}

\section{Notation and some useful facts \label{sec:Notation-and-some_useful_facts}}

This Section collects definitions, notation and previously known results
used in the paper. Notion related to simplicial complexes can be found
in Subsection \ref{subsec:Simplicial-complexes}. The definition of
the high-dimensional Laplacians and the boundary/coboudnary operators
appear in Subsection \ref{subsec:High-dimensional-Laplacians}. A
short summary on discrete Hodge theory is the content of Subsection
\ref{subsec:Discrete-Hodge-theory}. Finally, Subsection \ref{subsec:The--walk}
recalls the definition of the $\left(d-1\right)$-walk from \cite{PR12}
as well as some of the results proved there regarding its connection
to spectral and homological properties of the complex.

\subsection{Simplicial complexes\label{subsec:Simplicial-complexes}}

A simplicial complex $X$ is a collection of subsets of some countable
set $V$ that is closed under the operation of taking subsets. That
is, if $\tau\in X$ and $\sigma\subset\tau$ then $\sigma\in X$.
Elements of $X$ are called simplexes or cells and the dimension of
a simplex $\sigma\in X$ is defined to be $\left|\sigma\right|-1$.
A $j$-dimensional simplex is called a $j$-simplex or a $j$-cell.
The dimension of $X$, denoted by $d$, is defined to be $\max_{\sigma\in X}\mbox{dim}\left(\sigma\right)$.
A $d$-dimensional simplicial complex is called a $d$-complex for
short. We denote by $X^{j}$ the set of $j$-dimensional cells. The
degree of a $j$-cell, denoted $\mathrm{deg}\left(\sigma\right)$,
is the number of $\left(j+1\right)$-cells containing it and the set
of such $\left(j+1\right)$-cells, also known as its cofaces, is denoted
by $\mathrm{cf}\left(\sigma\right)=\left\{ \tau\in X^{j+1}\,:\,\sigma\subset\tau\right\} $. 

For $j\geq1$, each $j$-cell has two possible orientation, corresponding
to the ordering of its vertices up to an even permutation. Oriented
cells are denoted by square brackets; for example, the unoriented
$2$-cell $\left\{ u,v,w\right\} $ has two orientation $\left[u,v,w\right]=\left[w,u,v\right]=\left[v,w,u\right]$
and $\left[v,u,w\right]=\left[w,v,u\right]=\left[u,w,v\right]$. Given
an oriented cell $\sigma$ we denote by $\overline{\sigma}$ or $\left(-1\right)\sigma$
the same cell with the opposite orientation. The set of all oriented
$j$ cells is denoted by $X_{\pm}^{j}$. We also denote $X_{\pm}^{j}=X^{j}$
for $j=-1,0$. The faces of a $j$-cell $\sigma=\left\{ v_{0},\ldots,v_{j}\right\} $,
abbreviated $\mathrm{face}\left(\sigma\right)$, are the $\left(j-1\right)$-cells
$\left\{ \sigma\backslash v_{i}\right\} _{i=0}^{j}$. An oriented
$j$-cell $\sigma=\left[v_{0},\ldots,v_{j}\right]$, $j\geq2$ induces
an orientation on its faces given by $\left\{ \left(-1\right)^{i}\left[v_{0},\ldots,v_{i-1},v_{i+1},\ldots,v_{j}\right]\right\} _{i=0}^{j}$.
In a similar manner an oriented $j$-cell $\sigma$ induces an orientation
on its co-faces as follows: Given a cell $\sigma$ and a vertex $v\notin\sigma$
such that $v\sigma:=\left\{ v\right\} \cup\sigma$ is a coface of
$\sigma$ we write shortly $v\vartriangleleft\sigma$. If $\sigma=\left[\sigma_{0},\ldots,\sigma_{k}\right]$
is oriented and $v\vartriangleleft\sigma$, then $v\sigma$ inherits
the orientation $\left[v,\sigma_{0},\ldots,\sigma_{k}\right]$. 

The space of $j$-forms on $X$, denoted $\Omega^{j}=\Omega^{j}\left(X\right)$,
contains all function from $X_{\pm}^{j}$ to $\mathbb{R}$ which are
anti-symmetric with respect to a change of orientation, namely 
\[
\Omega^{j}=\left\{ f:X_{\pm}^{j}\to\mathbb{R}\Big|f\left(\overline{\sigma}\right)=-f\left(\sigma\right)\,\,\forall\sigma\in X_{\pm}^{j}\right\} .
\]
For $j=-1,0$ there are no orientations and thus $\Omega^{0}$ can
be identified with the space of functions on the vertices, while $\Omega^{-1}$
can naturally be identified with $\mathbb{R}$. To every $\sigma\in X_{\pm}^{j}$
one can associate a Dirac $j$-form $\ind_{\sigma}$ defined by 
\[
\ind_{\sigma}\left(\sigma'\right)=\begin{cases}
1 & \quad\sigma'=\sigma\\
-1 & \quad\sigma'=\overline{\sigma}\\
0 & \quad\mbox{otherwise}
\end{cases}.
\]

We also recall the following definitions from \cite{PR12}:
\begin{defn}[{Simplices neighboring relation \cite[Definition 2.1]{PR12}}]
\label{def:neighboring_relation} Let $1\leq j\leq d-1$. Two oriented
$j$-cells $\sigma,\sigma'\in X_{\pm}^{j}$ are called \emph{neighbors
}(denoted $\sigma\sim\sigma'$ or $\sigma\overset{\shortuparrow}{\sim}\sigma'$)
if $\sigma\cup\sigma'$ is a $\left(j+1\right)$-cell and the orientation
induced by $\sigma$ on $\sigma\cup\sigma'$ is opposite to the one
induced on it by $\sigma'$. In the case $j\geq2$ this is also equivalent
to the assumption that $\sigma\cup\sigma'\in X^{j+1}$, and that the
$\left(j-1\right)$-cell $\sigma\cap\sigma'$ inherits the same orientation
from both $\sigma$ and $\sigma'$. In the case $j=0$ the relation
$\sim$ is used to denote the usual graph neighboring relation, that
is $\sigma\sim\sigma'$ if both $0$-cells (vertices) are part of
a common $1$-cell (an edge). 
\end{defn}

\begin{defn}[{$k$-connectedness and disorientability \cite[Definitions 2.2, 2.6]{PR12}}]
$ $
\begin{enumerate}
\item Let $0\leq k\leq d-1$. We say that $X$ is $k$-connected if for
every pair of oriented $k$-cells $\sigma,\sigma'$ there exists a
chain $\sigma=\sigma_{0}\sim\sigma_{1}\sim\ldots\sim\sigma_{n}=\sigma'$.
Moreover, the existence of such a chain defines an equivalence relation
on the $k$-cells of $X$, whose equivalence classes are called the
$k$-components of $X$.
\item Let $0\leq k\leq d-1$. A \emph{$k$-disorientation }of a $d$-complex
is a choice of orientation $X_{+}^{k+1}$ for its $\left(k+1\right)$-cells,
so that whenever $\sigma,\sigma'\in X_{+}^{k+1}$ intersect in a $k$-cell
they induce the same orientation on it. If $X$ has a $k$-disorientation
it is said to be $k$-\emph{diorientable.}
\end{enumerate}
\end{defn}
A corresponding neighboring relation using faces instead of cofaces
was defined in \cite{MS13}.
\begin{defn}[{Simplices adjacency relation \cite[Definition 3.1]{MS13}}]
\label{def:adjacency_relation} Let $2\leq j\leq d$. Two oriented
$j$-cells $\sigma,\sigma'\in X_{\pm}^{j}$ are called adjacent (denoted
$\sigma\underset{\shortdownarrow}{\sim}\sigma'$) if $\sigma\cap\sigma'$
is a $\left(j-1\right)$-cell that inherits opposite orientations
from $\sigma$ and $\sigma'$. If $\sigma\cup\sigma'\in X^{j+1}$
this is equivalent to saying that they induce the same orientation
on their joint coface. In the case $j=1$, two oriented $1$-cells
(edges) $\sigma,\sigma'\in X_{\pm}^{1}$ are called adjacent if $\sigma\cap\sigma'$
is a vertex and exactly one of the edges points towards it. 
\end{defn}

\subsection{High dimensional Laplacians \label{subsec:High-dimensional-Laplacians}}

Let $X$ be a $d$-complex and $0\leq k\leq d$. The $k^{th}$ coboudnary
operator $\delta_{k}:\Omega^{k-1}\to\Omega^{k}$ is defined by 
\[
\delta_{k}f\left(\sigma\right)=\sum_{i=0}^{k}\left(-1\right)^{i}f\left(\sigma\backslash\sigma_{i}\right),\quad\forall f\in\Omega^{k}.
\]
\emph{Given a weight function $w:X\to\left(0,\infty\right)$, one
can introduce the Hilbert spaces 
\[
\Omega_{L^{2}}^{k}=\Omega_{L^{2}}^{k}\left(X\right)=\left\{ f\in\Omega^{k}\left(X\right)\,:\,\left\langle f,f\right\rangle <\infty\right\} ,
\]
with inner product 
\[
\left\langle f,g\right\rangle =\sum_{\sigma\in X^{k}}w\left(\sigma\right)f\left(\sigma\right)g\left(\sigma\right),\quad\forall f,g\in\Omega^{k}.
\]
Note that the sum is over unoriented $k$-cells, and that it is well
defined since the product $f\left(\sigma\right)g\left(\sigma\right)$
is independent of the orientation.}
\begin{claim}
\label{claim:k-good_1}Given a weight function $w:X\to\left(0,\infty\right)$.
The operator $\delta_{k}$ is bounded if and only if $\sup_{\sigma\in X^{k-1}}\frac{1}{w\left(\sigma\right)}\sum_{\tau\in\mathrm{cf}\left(\sigma\right)}w\left(\tau\right)<\infty$.
\end{claim}
Whenever $\delta_{k}$ is bounded its adjoint $\partial_{k}:=\delta_{k}^{*}:\Omega_{L^{2}}^{k}\to\Omega_{L^{2}}^{k-1}$
is defined by the relation $\left\langle \delta_{k}f,g\right\rangle =\left\langle f,\partial_{k}g\right\rangle $
for every $f\in\Omega_{L^{2}}^{k-1}$ and every $g\in\Omega_{L^{2}}^{k}$.
One can verify that in this case  
\begin{equation}
\partial_{k}g\left(\sigma\right)=\frac{1}{w\left(\sigma\right)}\sum_{v\vartriangleleft\sigma}w\left(v\sigma\right)g\left(v\sigma\right).\label{eq:boundary_operator}
\end{equation}
The last equation can be taken as the definition of $\partial_{k}$
even when the required assumptions on $\delta_{k}$ are not satisfied,
however in this case $\partial_{k}:\Omega_{L^{2}}^{k}\to\Omega^{k-1}$
is not necessarily well defined since $\deg\left(\sigma\right)$ might
be infinite. 
\begin{claim}
\label{claim:k-good_2}If $\deg\left(\sigma\right)<\infty$ for every
$\sigma\in X^{k-1}$ then $\partial_{k}$ is well defined. In addition,
the operator $\partial_{k}$ is bounded whenever $\sup_{\sigma\in X^{k-1}}\frac{1}{w\left(\sigma\right)}\sum_{\tau\in\mathrm{cf}\left(\sigma\right)}w\left(\tau\right)<\infty$. 
\end{claim}
The last two claims portend the following definition:
\begin{defn}
\label{def:k-good_weight_function}A weight function $w:X\to\left(0,\infty\right)$
is called $k$-good if 
\begin{equation}
\sup_{\sigma\in X^{k-1}}\frac{1}{w\left(\sigma\right)}\sum_{\tau\in\mathrm{cf}\left(\sigma\right)}w\left(\tau\right)<\infty.\label{eq:good_weight_functions}
\end{equation}
If $w$ is $k$-good for every $0\leq k\leq d$ we simply say that
$w$ is good. \end{defn}
\begin{example}
\label{exa:weight_functions}$ $
\begin{enumerate}
\item If $X$ is a finite $d$-complex, then every weight function $w:X\to\left(0,\infty\right)$
is good.
\item Assume that $X$ is a $d$-complex such that $1\leq\deg\left(\sigma\right)<\infty$
for every $\sigma\in X^{d-1}$ and let $w_{\shortuparrow}:X\to\left(0,\infty\right)$
be the weight function 
\[
w_{\shortuparrow}\left(\sigma\right)=\begin{cases}
\deg\left(\sigma\right) & \,\,\sigma\in X^{d-1}\\
1 & \,\,\sigma\notin X^{d-1}
\end{cases}.
\]
Then for $\sigma\in X^{k-1}$ 
\[
\frac{1}{w_{\shortuparrow}\left(\sigma\right)}\sum_{\tau\in\mathrm{cf}\left(\sigma\right)}w_{\shortuparrow}\left(\tau\right)=\begin{cases}
1 & \,\, k=d\\
\sum_{\tau\in\mathrm{cf}\left(\sigma\right)}\deg\left(\tau\right) & \,\, k=d-1\\
\deg\left(\sigma\right) & \,\, k<d-1
\end{cases}.
\]
Therefore $w_{\shortuparrow}$ is always $d$-good, is $\left(d-1\right)$-good
if and only if $\sup_{\sigma\in X^{d-2}}\sum_{\tau\in\mathrm{cf}\left(\sigma\right)}\deg\left(\tau\right)<\infty$
and is $k$-good for $k<d-1$ if and only if the degrees of the $\left(k-1\right)$-cells
are uniformly bounded. 
\item Assume that $X$ is a $d$-complex such that $\deg\left(\sigma\right)<\infty$
for every $\sigma\in X^{d-1}$ and let $w_{\shortdownarrow}:X\to\left(0,\infty\right)$
be the weight function 
\[
w_{\shortdownarrow}\left(\sigma\right)=\begin{cases}
\frac{1}{d+1} & \,\,\sigma\in X^{d}\\
1 & \,\,\sigma\notin X^{d}
\end{cases}.
\]
Then for $\sigma\in X^{k-1}$ 
\[
\frac{1}{w_{\shortdownarrow}\left(\sigma\right)}\sum_{\tau\in\mathrm{cf}\left(\sigma\right)}w_{\shortdownarrow}\left(\tau\right)=\begin{cases}
\frac{\deg\left(\sigma\right)}{d+1} & \,\, k=d\\
\deg\left(\sigma\right) & \,\, k<d
\end{cases}.
\]
Therefore $w_{\shortdownarrow}$ is $k$-good if and only if the degrees
of the $\left(k-1\right)$-cells are uniformly bounded. 
\end{enumerate}
\end{example}
Whenever the operators $\partial_{\cdot}$ and $\delta_{\cdot}$ are
well defined, the \emph{upper, lower and full Laplacians,} $\Delta_{k}^{+},\Delta_{k}^{-}:\Omega_{L^{2}}^{k}\to\Omega^{k}$
and $\Delta_{k}:\Omega_{L^{2}}^{k}\to\Omega^{k}$ respectively, are
given by 
\begin{align*}
 & \Delta_{k}^{+}=\partial_{k+1}\delta_{k+1} &  & \!\!\!\!\!-1\leq k\leq d-1,\\
 & \Delta_{k}^{-}=\delta_{k}\partial_{k} &  & \,\,0\leq k\leq d,\\
 & \Delta_{k}=\Delta_{k}^{+}+\Delta_{k}^{-} &  & \,\,0\leq k\leq d-1.
\end{align*}
The special case of $\Delta_{d-1}^{+}$ will be abbreviated $\Delta^{+}$.

A short calculation gives  
\begin{equation}
\Delta_{k}^{+}f\left(\sigma\right)=\frac{1}{w\left(\sigma\right)}\left(\sum_{\tau\in\mathrm{cf}\left(\sigma\right)}w\left(\tau\right)\right)f\left(\sigma\right)-\frac{1}{w\left(\sigma\right)}\sum_{\sigma'\sim\sigma}w\left(\sigma'\cup\sigma\right)f\left(\sigma'\right)\label{eq:upp_Laplacian}
\end{equation}
and 
\begin{equation}
\Delta_{k}^{-}f\left(\sigma\right)=\left(\sum_{\tau\in\mathrm{face}\left(\sigma\right)}\frac{w\left(\sigma\right)}{w\left(\tau\right)}\right)f\left(\sigma\right)-\sum_{\sigma'\underset{\shortdownarrow}{\sim}\sigma}\frac{w\left(\sigma'\right)}{w\left(\sigma\cap\sigma'\right)}f\left(\sigma'\right).\label{eq:down_Laplacian}
\end{equation}
 The space of \emph{Harmonic $k$-forms,} denoted $\mathcal{H}^{k}=\mathcal{H}^{k}\left(X\right)$,
is defined to be the kernel of $\Delta_{k}$. 

Throughout the paper (except for Section \ref{sec:Lower---branching_random_walk})
the weight function $w_{\shortuparrow}$ from Example \ref{exa:weight_functions}(2)
is used, in which case one gets 
\begin{equation}
\Delta^{+}f\left(\sigma\right)=f\left(\sigma\right)-\frac{1}{\deg\left(\sigma\right)}\sum_{\sigma'\sim\sigma}f\left(\sigma'\right).\label{eq:Up_Laplacian_that_we_use}
\end{equation}

\subsection{Discrete Hodge theory \label{subsec:Discrete-Hodge-theory}}

The sequence $\left(\Omega^{k},\delta_{k+1}\right)$ is a simplicial
cochain complex of $X$, meaning that $\delta_{k+1}\delta_{k}=0$
for every $k$. The chain structure gives rise to the 
\begin{align*}
 & \,\,\,\mbox{\emph{k-cocycles}}\,\,(closed\, forms) &  & Z^{k}=\ker\delta_{k+1},\\
 & \,\,\,\mbox{\emph{k-coboundaries}}\,\,(exact\, forms) &  & B^{k}=\mathrm{im}\delta_{k},\\
 & \,\,\,\mbox{\emph{k-cohomology}} &  & H^{k}=\nicefrac{Z^{k}}{B^{k}}.
\end{align*}

When $X$ is a \textbf{finite} complex (or more generally when $w$
is a good weight function) the sequence $\left(\Omega^{k},\partial_{k}\right)$
is a simplicial chain complex of $X$ and this gives rise to 
\begin{align*}
 & \mbox{\emph{k-cycles}} & \phantom{\quad\quad\quad\quad\quad\quad\,\,\,} & Z_{k}=\ker\partial_{k},\\
 & \mbox{\emph{k-boundaries}} &  & B_{k}=\mathrm{im}\partial_{k+1},\\
 & \mbox{\emph{k-homology}} &  & H_{k}=\nicefrac{Z_{k}}{B_{k}}.
\end{align*}
The isomorphism between harmonic $k$-forms, the $k$-cohomology and
the $k$-homology as well as the connection to the boundary operators
is known as discrete Hodge theorem. In the discrete setting it originates
in the Work of Eckmann \cite{Eck44} and is summarized in the following
lemma:
\begin{lem}[Discrete Hodge theory \cite{Eck44}]
 Let $X$ be a \textbf{finite} $d$-complex. Then for any $-1\leq k\leq d-1$
and any weight function $w:X\to\left(0,\infty\right)$
\begin{enumerate}
\item $Z_{k}=\ker\partial_{k}=\ker\Delta_{k}^{-}=\left(B^{k}\right)^{\bot}$.
\item $B_{k}=\mathrm{im}\partial_{k+1}=\mathrm{im}\Delta_{k}^{+}=\left(Z^{k}\right)^{\bot}$.
\item $Z^{k}=\ker\delta_{k+1}=\ker\Delta_{k}^{+}=\left(B_{k}\right)^{\bot}$.
\item $B^{k}=\mathrm{im}\delta_{k}=\mathrm{im}\Delta_{k}^{-}=\left(Z_{k}\right)^{\bot}$.
\item $\mathcal{H}^{k}=\ker\Delta_{k}=Z_{k}\cap Z^{k}=\left(B_{k}\oplus B^{k}\right)^{\bot}\cong H_{k}\cong H^{k}.$
\item $\Omega^{k}=\rlap{\ensuremath{\overbrace{\phantom{B_{k}\oplus\mathcal{H}^{k}}}^{Z_{k}}}}B_{k}\oplus\underbrace{\mathcal{H}^{k}\oplus B^{k}}_{Z^{k}}$
(Hodge decomposition).
\end{enumerate}
\end{lem}
Due to the chain complex structure $B_{k}\subset Z_{k}$, which implies
that the Laplacian $\Delta_{k}^{+}$ always has trivial zeroes in
its kernel. The \emph{spectral gap} of a finite $d$-complex $X$,
denoted $\lambda_{k}\left(X\right)$, is defined to be the smallest
non-trivial eigenvalue of $\Delta_{k}^{+}$ that is 
\[
\lambda_{k}\left(X\right)=\min\left(\mathrm{Spec}\left(\Delta_{k}^{+}\Big|_{\left(B^{k-1}\right)^{\bot}}\right)\right)=\min\left(\mathrm{Spec}\left(\Delta_{k}^{+}\Big|_{Z_{k-1}}\right)\right).
\]

Going back to the case of a general $d$-complex $X$, the boundary
operators $\partial_{k}$ are not well defined and even when they
are it is possible to have $\partial_{k}\partial_{k+1}\left(f\right)\neq0$
for some $f\in\Omega_{L^{2}}^{k+1}$. If $w:X\to\left(0,\infty\right)$
is both $k$ and $\left(k+1\right)$-good then both operators are
well defined and bounded and due to the fact that $\partial_{k}^{*}=\delta_{k}$
and $\partial_{k+1}=\delta_{k+1}^{*}$ it follows that $\partial_{k}\partial_{k+1}=0$.
The interested reader might want to consult \cite{PR12} for additional
discussion on the general case.

\subsection{The $\left(d-1\right)$-walk \label{subsec:The--walk}}

In this subsection we recall the definition of the $\left(d-1\right)$-walk
constructed in \cite{PR12} as well as some of its properties.
\begin{defn}
\cite[Definition 2.1]{PR12} The $p$-lazy $\left(d-1\right)$-walk
on $X$ is a time-homogeneous Markov chain with state space $X_{\pm}^{d-1}$
that stays put with probability $p$, and with probability $\left(1-p\right)$
chooses one of its neighbors (see Definition \ref{def:neighboring_relation})
in $X_{\pm}^{d-1}$ uniformly at random and jumps to it. More formally,
this is a Markov chain $\left(Y_{n}\right)_{n\geq0}$ with state space
$X_{\pm}^{d-1}$ and transition probabilities 
\[
\mathrm{Prob}\left(Y_{n+1}=\sigma'\Big|Y_{n}=\sigma\right)=\begin{cases}
p & \quad\sigma'=\sigma\\
\frac{1-p}{d\cdot\deg\left(\sigma\right)} & \quad\sigma'\sim\sigma\\
0 & \quad\mbox{otherwise}
\end{cases}.
\]

\end{defn}
The heat kernel of the random walk $\left(\mathbf{p}_{n}\left(\sigma,\sigma'\right)\right)_{n\geq0,\,\sigma,\sigma'\in X_{\pm}^{d-1}}$
is defined by 
\[
\mathbf{p}_{n}\left(\sigma,\sigma'\right)=\mathrm{Prob}\left(Y_{n}=\sigma'\Big|Y_{0}=\sigma\right).
\]

The behavior of the $\left(d-1\right)$-random walk, or more precisely
of its heat kernel, relates to the $\left(d-1\right)$-connectedness
of the complex in the same way that a classic random walk on a graph
relates to the connectedness of the graph. In order to relate the
$\left(d-1\right)$-walk to the homology and cohomology of the complex,
which are more natural counterparts of connectedness in high dimensions,
the authors introduced the \emph{expectation process} $\mathcal{E}_{n}:X_{\pm}^{d-1}\times X_{\pm}^{d-1}\to\left[-1,1\right]$
which for $d\geq2$ is defined by%
\footnote{In the case $d=1$ the expectation process is simply defined to be
heat kernel.%
} 

\[
\mathcal{E}_{n}\left(\sigma,\sigma'\right)=\mathbf{p}_{n}\left(\sigma,\sigma'\right)-\mathbf{p}_{n}\left(\sigma,\overline{\sigma'}\right).
\]
Unfortunately a new problem arises when observing the expectation
process, that is $\lim_{n\to\infty}\mathcal{E}_{n}\left(\sigma,\sigma'\right)=0$,
for every $\sigma,\sigma'\in X_{\pm}^{d-1}$. However, it was proven
in \cite{PR12} that $\left|\mathcal{E}_{n}\left(\sigma,\sigma'\right)\right|=\Theta\left(\left(\frac{p\left(d-1\right)+1}{d}\right)^{n}\right)$
for every finite $d$-complex $X$, which called upon the definition
of a \emph{normalized expectation process}

\[
\widetilde{\mathcal{E}}_{n}\left(\sigma,\sigma'\right)=\left(\frac{d}{p\left(d-1\right)+1}\right)^{n}\mathcal{E}_{n}\left(\sigma,\sigma'\right).
\]

The evolution of the expectation process and its normalized version
in time is given by $\mathcal{E}_{n+1}\left(\sigma,\cdot\right)=\left(A_{p}\mathcal{E}_{n}\right)\left(\sigma,\cdot\right)$
and $\widetilde{\mathcal{E}}_{n+1}\left(\sigma,\cdot\right)=\left(\left(\frac{d}{p\left(d-1\right)+1}\right)A_{p}\widetilde{\mathcal{E}}_{n}\right)\left(\sigma,\cdot\right)$
respectively, where 
\begin{equation}
\left(A_{p}f\right)\left(\sigma\right)=pf\left(\sigma\right)+\frac{\left(1-p\right)}{d}\sum_{\sigma'\sim\sigma}\frac{f\left(\sigma'\right)}{\deg\left(\sigma'\right)},\label{eq:prop_operator_(d-1)-walk}
\end{equation}
and $A_{p}$ acts on the second coordinate.

The following theorem summarizes the connection between the asymptotics
of the normalized expectation process and the homology of the a complex:
\begin{thm}[{\cite[Theorem 2.9 and (2.1)]{PR12}}]
\label{thm:(d-1)-walk_vs_homology}Let $X$ be a \textbf{finite $d$}-complex
and $\widetilde{\mathcal{E}}_{n}$ the normalized expectation process
associated with the $p$-lazy $\left(d-1\right)$-walk on $X$. 
\begin{enumerate}
\item If $\frac{d-1}{3d-1}<p<1$, then $\widetilde{\mathcal{E}}_{\infty}=\lim_{n\to\infty}\widetilde{\mathcal{E}}_{n}$
always exists. In addition if $p=\frac{d-1}{3d-1}$ then $\widetilde{\mathcal{E}}_{\infty}=\lim_{n\to\infty}\widetilde{\mathcal{E}}_{n}$
exists whenever $X$ has no disorientable $\left(d-1\right)$-components. 
\item If $p>\frac{d-1}{3d-1}$ or $p=\frac{d-1}{3d-1}$ and $X$ has no
disorientable $\left(d-1\right)$-components, then $\left\{ \widetilde{\mathcal{E}}_{\infty}\left(\sigma,\cdot\right)\right\} _{\sigma\in X_{\pm}^{d-1}}\subset B^{d-1}$
if and only if $H_{d-1}\left(X\right)=0$. 
\item More generally, If $p>\frac{d-1}{3d-1}$ or $p=\frac{d-1}{3d-1}$
and $X$ has no disorientable $\left(d-1\right)$-components, then
the dimension of $H_{d-1}\left(X\right)$ equals the dimension of
$\mathrm{Span}\left\{ \mathrm{Proj}_{Z_{d-1}}\left(\widetilde{\mathcal{E}}_{\infty}\left(\sigma,\cdot\right)\right)\,:\,\sigma\in X_{\pm}^{d-1}\right\} $,
where $\mathrm{Proj}_{Z_{d-1}}$ is the orthogonal projection in $\Omega^{d-1}$
onto $Z_{d-1}$.
\item If furthermore $p\geq\frac{1}{2}$ then 
\[
\mathrm{dist}\left(\widetilde{\mathcal{E}}_{n},B^{d-1}\right)=O\left(\left(1-\frac{1-p}{p\left(d-1\right)+1}\lambda_{d-1}\left(X\right)\right)^{n}\right).
\]

\end{enumerate}
\end{thm}
\begin{rem}
When necessary the notation $\mathcal{E}_{n}^{p}$ and $\widetilde{\mathcal{E}}_{n}^{p}$
is used to stress the dependence of $\mathcal{E}_{n}$ and $\widetilde{\mathcal{E}}_{n}$
on $p$.
\end{rem}

\section{Simplicial branching random walks \label{sec:The--branching-random-walk}}

This section is devoted to the definition of simplicial branching
randoms walk and its effective version as well as to the study of
their basic properties. In the first part, the definition of the processes
is given and the first result (Theorem \ref{thm:Finite_complexes_BRW_vs_homology})
is proved. In the second part, the associated tree structure is described
and a high-dimensional version of (\ref{eq:return_prob_RW}) is proved
(see Theorem \ref{thm:R_via_S_and_overline=00007BS=00007D}). A discussion
on several possible variations of the model can be found in Remark
\ref{rem:variants_on_the_process}. Throughout this section $X$ denotes
a $d$-complex such that $1\leq\deg\left(\sigma\right)<\infty$ for
every $\sigma\in X^{d-1}$.

\medskip{}

The $p$-lazy simplicial branching random walk on $X$ is a time-homogeneous
Markov chain $\left(N_{n}\left(\cdot\right)\right)_{n\geq0}$ with
state space $\mathbb{N}^{X_{\pm}^{d-1}}$ which describes the number
of particles at time $n$ on any of the oriented $\left(d-1\right)$-cells,
that is:
\begin{itemize}
\item $N_{n}$ is a random function from $X_{\pm}^{d-1}$ to $\mathbb{N}$.
\item The process is Markovian, i.e., $\mathrm{Prob}\left(N_{n}\in A\Big|N_{1},\ldots,N_{n-1}\right)=\mathrm{Prob}\left(N_{n}\in A\Big|N_{n-1}\right)$
and time homogeneous, namely $\mathrm{Prob}\left(N_{n}=g\Big|N_{n-1}=f\right)$
doesn't depend on $n$.
\item $N_{n}\left(\sigma\right)$ is the random number of particles in $\sigma$
at time $n$ for every $\sigma\in X_{\pm}^{d-1}$ and $n\geq0$. 
\end{itemize}
One step evolution of the process (its transition kernel) is defined
as follows: Given a configuration of particles on $X_{\pm}^{d-1}$
all the particles evolve simultaneously and independently. If a particle
is positioned in $\sigma$, then it stays put with probability $p$,
and with probability $1-p$ chooses one of the cofaces of $\sigma$
uniformly at random and splits into $d$ new particles which are now
positioned on the neighbors of $\sigma$ in the chosen coface (one
on each such neighbor). Note that one step of the process is comprised
of the evolution of all existing particles. An illustration of one
step of the process on a triangle complex can be found in Figure \ref{fig:One_step_of_the_branching_process}.

One way to realize the process is as follows: Let $\left(\eta_{\sigma}\right)_{\sigma\in X_{\pm}^{d-1}}$
be random variables taking values in $X^{d}\uplus\left\{ \zeta\right\} $,
with each $\eta_{\sigma}$ distributed like 
\[
Prob\left(\eta_{\sigma}=\tau\right)=\begin{cases}
p & \,\,,\tau=\zeta\\
\frac{1-p}{\deg\left(\sigma\right)} & \,\,,\tau\in\mathrm{cf}\left(\sigma\right)
\end{cases}.
\]
Then $\left(N_{n}\right)_{n\geq0}$ is a Markovian process, taking
values in $\mathbb{N}^{X_{\pm}^{d-1}}$ defined by 
\[
N_{n+1}\left(\sigma\right)=\sum_{i=1}^{N_{n}\left(\sigma\right)}\ind_{\xi_{\sigma}^{i,n}=\zeta}+\sum_{\sigma'\sim\sigma}\sum_{i=1}^{N_{n}\left(\sigma'\right)}\ind_{\xi_{\sigma}^{i,n}=\sigma\cup\sigma'},
\]
where $\left(\eta_{\sigma}^{i,n}\right)_{i\geq1,n\geq0}$ are i.i.d.
copies of $\eta_{\sigma}$. 

For a given $\pi\in\mathbb{N}^{X_{\pm}^{d-1}}$ we denote by $P^{\pi}$
the distribution of $\left(N_{n}\right)_{n\geq0}$ with the above
law and starting distribution $P^{\pi}\left(N_{0}=\pi\right)=1$.
The expectation with respect to $P^{\pi}$ is denoted by $E^{\pi}$.
In the case $\pi=\delta_{\sigma}$, where $\delta_{\sigma}\left(\sigma'\right)=\begin{cases}
1 & \sigma'=\sigma\\
0 & \mbox{otherwise}
\end{cases}$, we abbreviate $P^{\sigma}$ and $E^{\sigma}$ instead of $P^{\delta_{\sigma}}$
and $E^{\delta_{\sigma}}$. 

The process which is truly the source of our interest is not $\left(N_{n}\right)_{n\geq0}$
but rather its effective version defined by 
\begin{equation}
D_{n}\left(\sigma\right)=N_{n}\left(\sigma\right)-N_{n}\left(\overline{\sigma}\right),\quad\forall\sigma\in X_{\pm}^{d-1}.\label{eq:EBRW}
\end{equation}
 Note that $\left(D_{n}\right)_{n\geq0}$ is a sequence of random
forms in $\Omega^{d-1}$. 

Finally, the heat kernel of $\left(D_{n}\right)_{n\geq0}$ is defined
by 
\[
\mathscr{E}_{n}\left(\sigma,\sigma'\right)=E^{\sigma}\left[D_{n}\left(\sigma'\right)\right],\quad\forall\sigma,\sigma'\in X_{\pm}^{d-1}.
\]
When necessary the notation $\mathscr{E}_{n}^{p}\left(\sigma,\sigma'\right)$
will be used to stress the dependence on $p$.

We are now ready to give a formal statement of our first result:
\begin{thm}
\label{thm:Finite_complexes_BRW_vs_homology} Let $X$ be a \textbf{finite
$d$}-complex, $\left(D_{n}\right)_{n\geq0}$ the $p$-lazy ESBRW
on $X$ and $\left(\mathscr{E}_{n}\right)_{n\geq0}$ its heat kernel. 
\begin{enumerate}
\item If $\frac{d-1}{d+1}<p<1$, then $\mathscr{E}_{\infty}=\lim_{n\to\infty}\mathscr{E}_{n}$
always exists. In addition if $p=\frac{d-1}{d+1}$ then $\mathscr{E}_{\infty}=\lim_{n\to\infty}\mathscr{E}_{n}$
exists whenever $X$ has no disorientable $\left(d-1\right)$-components. 
\item If $p>\frac{d-1}{d+1}$ or $p=\frac{d-1}{d+1}$ and $X$ has no disorientable
$\left(d-1\right)$-components, then $\left\{ \mathscr{E}_{\infty}\left(\sigma,\cdot\right)\right\} _{\sigma\in X_{\pm}^{d-1}}\subset B^{d-1}$
if and only if $H_{d-1}\left(X\right)=0$. 
\item More generally, if $p>\frac{d-1}{d+1}$ or $p=\frac{d-1}{d+1}$ and
$X$ has no disorientable $\left(d-1\right)$-components, then the
dimension of $H_{d-1}\left(X\right)$ equals the dimension of $\mathrm{Span}\left\{ \mathrm{Proj}_{Z_{d-1}}\left(\mathscr{E}_{\infty}\left(\sigma,\cdot\right)\right)\,:\,\sigma\in X_{\pm}^{d-1}\right\} $,
where $\mathrm{Proj}_{Z_{d-1}}$ is the orthogonal projection in $\Omega^{d-1}$
onto $Z_{d-1}$.
\item If furthermore $p\geq\frac{d}{d+1}$ then 
\[
\mathrm{dist}\left(\mathscr{E}_{n},B^{d-1}\right)=O\left(\left(1-\left(1-p\right)\lambda_{d-1}\left(X\right)\right)^{n}\right).
\]

\end{enumerate}
\end{thm}
The following lemma contains the main ingredient for the proof of
Theorem \ref{thm:Finite_complexes_BRW_vs_homology} besides Theorem
\ref{thm:(d-1)-walk_vs_homology}:
\begin{lem}[Time evolution of the heat kernel and its connection to the expectation
process]
\label{lem:comparison_of_models}$ $ 
\begin{enumerate}
\item For every $0\leq p\leq1$ and $\sigma'\in X_{\pm}^{d-1}$ the evolution
of $\mathscr{E}_{n}^{p}\left(\cdot,\sigma'\right)$ in time is given
by $\mathscr{E}_{n}^{p}\left(\cdot,\sigma'\right)=\left(\mathscr{A}_{p}\mathscr{E}_{n-1}^{p}\right)\left(\cdot,\sigma'\right)$,
where $\mathscr{A}_{p}:\Omega^{d-1}\to\Omega^{d-1}$ acts on the first
coordinate and is given by 
\[
\mathscr{A}_{p}f\left(\sigma\right)=pf\left(\sigma\right)+\frac{1-p}{\deg\left(\sigma\right)}\sum_{\sigma'\sim\sigma}f\left(\sigma'\right)=\left(I-\left(1-p\right)\Delta^{+}\right)f\left(\sigma\right),\quad\forall f\in\Omega_{L^{2}}^{d-1}.
\]

\item For every $0\leq p\leq1$ 
\[
\mathscr{E}_{n}^{p}\left(\sigma,\sigma'\right)=\widetilde{\mathcal{E}}_{n}^{p'}\left(\sigma,\sigma'\right),\quad\forall\sigma,\sigma'\in X_{\pm}^{d-1},
\]
where $p'=\frac{p}{1+\left(1-p\right)\left(d-1\right)}$. 
\end{enumerate}
\end{lem}
\begin{proof}
$ $
\begin{enumerate}
\item By the Markov property, for every $n\geq1$
\begin{align}
\mathscr{E}_{n}\left(\sigma,\sigma'\right) & =E^{\sigma}\left[E^{\sigma}\left[D_{n}\left(\sigma'\right)\Big|N_{1}\right]\right]=E^{\sigma}\left[E^{N_{1}}\left[D_{n-1}\left(\sigma'\right)\right]\right]\nonumber \\
 & =E^{\sigma}\left[\sum_{\sigma''\in X_{\pm}^{d-1}}N_{1}\left(\sigma''\right)\cdot E^{\sigma''}\left[D_{n-1}\left(\sigma'\right)\right]\right]=\sum_{\sigma''\in X_{\pm}^{d-1}}E^{\sigma}\left[N_{1}\left(\sigma''\right)\right]\mathscr{E}_{n-1}\left(\sigma'',\sigma'\right)\nonumber \\
 & =p\mathscr{E}_{n-1}\left(\sigma,\sigma'\right)+\frac{1-p}{\deg\left(\sigma\right)}\cdot\sum_{\sigma''\sim\sigma}\mathscr{E}_{n-1}\left(\sigma'',\sigma'\right).\label{eq:one_step_evolution_of_mathscr=00007BE=00007D}
\end{align}

\item Using part (1) and the equality $\mathscr{E}_{0}\left(\sigma,\sigma'\right)=\ind_{\sigma}\left(\sigma'\right)$
it follows that $\mathscr{E}_{n}^{p}\left(\sigma,\sigma'\right)=\deg\left(\sigma'\right)\cdot\left\langle \mathscr{A}_{p}^{n}\ind_{\sigma},\ind_{\sigma'}\right\rangle $
for every $n\geq0$. Similarly, by (\ref{eq:prop_operator_(d-1)-walk})
\begin{align*}
\widetilde{\mathcal{E}}_{n}^{p'}\left(\sigma,\sigma'\right) & =\left(\frac{d}{p'\left(d-1\right)+1}\right)^{n}\mathcal{E}_{n}^{p'}\left(\sigma,\sigma'\right)=\deg\left(\sigma'\right)\cdot\left(\frac{d}{p'\left(d-1\right)+1}\right)^{n}\left\langle \ind_{\sigma},A_{p'}^{n}\ind_{\sigma'}\right\rangle \\
 & =\deg\left(\sigma'\right)\cdot\left\langle \left(\left(\left(\frac{d}{p'\left(d-1\right)+1}\right)\right)\cdot A_{p'}^{tr}\right)^{n}\ind_{\sigma},\ind_{\sigma'}\right\rangle 
\end{align*}
 for every $n\geq0$, where $A_{p}^{tr}$ is the transpose of $A_{p}$.
The claim now follows since 
\[
\mathscr{A}_{p}=I-\left(1-p\right)\Delta^{+}=\left(\frac{d}{p'\left(d-1\right)+1}\right)\cdot A_{p'}^{tr},
\]
for $p'=\frac{p}{1+\left(1-p\right)\left(d-1\right)}$. 
\end{enumerate}
\end{proof}

\begin{proof}[Proof of Theorem \ref{thm:Finite_complexes_BRW_vs_homology}]
 The proof follows by combining Theorem \ref{thm:(d-1)-walk_vs_homology}
and Lemma \ref{lem:comparison_of_models}(2).\end{proof}
\begin{rem}[Variants on the model ]
\label{rem:variants_on_the_process} Before turning to discuss the
tree structure associated with the SBRW we wish to introduce some
possible variants for the model. First, due to the fact that our main
interest lies in the ESBRW $\left(D_{n}\right)_{n\geq0}$ and not
in the SBRW itself, it is possible to annihilate any pair of particles
on the same cell with different orientation. That is, if at time $n$
there are $N_{n}\left(\sigma\right)=k_{1}$ and $N_{n}\left(\overline{\sigma}\right)=k_{2}$
particles of type $\sigma$ and $\overline{\sigma}$ respectively,
and without loss of generality $k_{1}\geq k_{2}$, then all of them
annihilates except for $k_{1}-k_{2}$ of the $\sigma$ particles.
This variant on the model is nothing else than a different choice
of coupling for the branching of the particles. Indeed, for this choice
any pair of particles on the same cell with different orientation
are coupled to branch together. Other couplings can also be considered.
Secondly, in order to avoid simultaneous splitting of the particles
one can work with a continuous time version where each particle has
a Poisson clock to determine its branching time. Finally, note that
the above model can also be generalized to give a high-dimensional
analogue of weighted random walk on graphs by considering other weight
functions. 

In addition, one can also consider the above process on $k$-oriented
cells of a $d$-complex for every $0\leq k\leq d-1$ and not just
for $k=d-1$. This however is equivalent to studying the original
process on $X^{k+1}$ and therefore falls back to the above setting. 
\end{rem}

\subsection{Expected number of first visits}

The SBRW is a fusion between a multi-type branching process and a
random walk. On the one hand in every step the current population
of particles splits and creates a new population in the same manner
as in a branching process. On the other hand the law that specify
the siblings of each particle is governed by the law of a $\left(d-1\right)$-random
walk on the complex. 

To every branching process, and in particular the SBRW, one can associate
a natural tree structure where the siblings of each particle are the
one generated from it (see Figure \ref{fig:The_branching_process_and_the_tree}
for an illustration). The tree structure also allows us to associate
with each particle a unique sequence of ancestors. These facts are
summarized in the following definition: 
\begin{defn}
$ $
\begin{enumerate}
\item For $\sigma\in X_{\pm}^{d-1}$ and $n\geq0$ let $\Psi_{n}\left(\sigma\right)$
be the set of particles in $\sigma$ at time $n$. Note that $N_{n}\left(\sigma\right)=\left|\Psi_{n}\left(\sigma\right)\right|$.
\item Denote $\Psi_{n}=\biguplus_{\sigma\in X_{\pm}^{d-1}}\Psi_{n}\left(\sigma\right)$,
the set of all particles at time $n$. 
\item To each element in the set of particles at time $n$ one can associate
a unique sequence of ancestors going back to the set of particles
at time $0$. Given $\xi\in\Psi_{n}$ denote by $\mathfrak{a}\xi$
the unique ancestor of $\xi$ in $\Psi_{n-1}$. Continuing recursively
one can define $\mathfrak{a}^{k}\xi=\mathfrak{a}\left(\mathfrak{a}^{k-1}\xi\right)$
for every $2\leq k\leq n$. 
\end{enumerate}
\end{defn}
\begin{figure}[h]
\centering{}\includegraphics[scale=1.1]{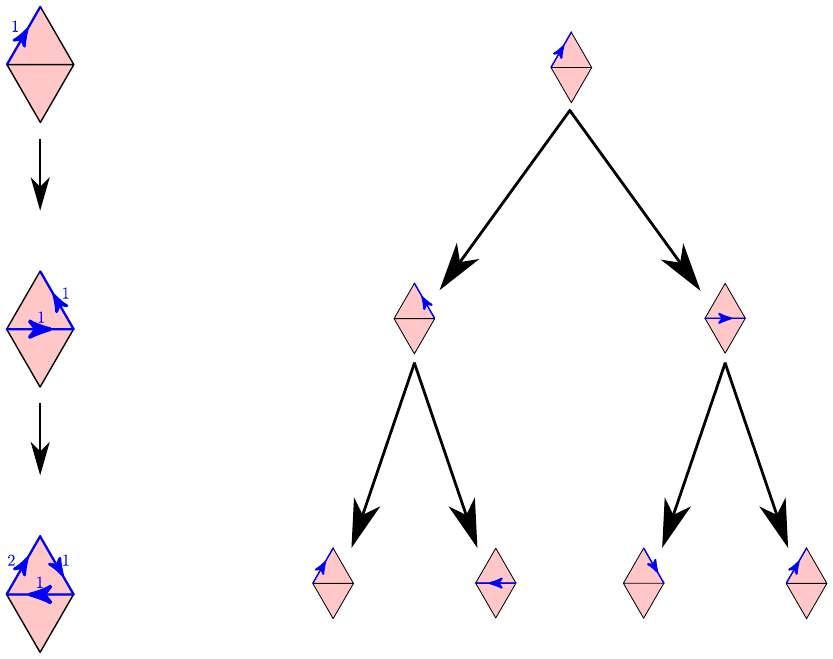}\protect\caption{Two steps of the branching random walk and the tree associated with
it.\label{fig:The_branching_process_and_the_tree}}
\end{figure}

The definition of ancestors of a particle allows us to generalize
the important notion of first return to a vertex from random walks
on graphs:
\begin{defn}
\label{def:expected_num_of_first_visits}For $\sigma\in X_{\pm}^{d-1}$
and $n\geq1$ define $K_{n}\left(\sigma\right)$ to be the number
of particles in $\sigma$ at time $n$ that none of their ancestors
(except perhaps to the one at time zero) were in $\sigma$ or $\overline{\sigma}$.
Namely, 
\[
K_{n}\left(\sigma\right)=\#\left\{ \xi\in\Psi_{n}\left(\sigma\right)\,:\,\mathfrak{a}^{k}\xi\notin\Psi_{n-k}\left(\sigma\right)\cup\Psi_{n-k}\left(\overline{\sigma}\right)\,\,\forall1\leq k<n\right\} .
\]
For $n=0$ define $K_{0}\left(\sigma\right)=0$. We also denote $F_{n}\left(\sigma\right)=K_{n}\left(\sigma\right)-K_{n}\left(\overline{\sigma}\right)$
and $\mathscr{F}_{n}\left(\sigma,\sigma'\right)=E^{\sigma}\left[F_{n}\left(\sigma'\right)\right]$. 
\end{defn}
The main goal of this subsection is to prove the following result:
\begin{thm}
\label{thm:R_via_S_and_overline=00007BS=00007D}Let $X$ be a $d$-complex,
$\sigma\in X_{\pm}^{d-1}$ and $z\in\mathbb{C}$. Define the power
series 
\begin{enumerate}
\item $\mathcal{G}\left(z\right)=\sum_{n=0}^{\infty}\mathscr{E}_{n}\left(\sigma,\sigma\right)z^{n}$ 
\item $\mathfrak{F}\left(z\right)=\sum_{n=0}^{\infty}\mathscr{F}_{n}\left(\sigma,\sigma\right)z^{n}$
\end{enumerate}

Then, for every $z\in\mathbb{C}$ whose absolute value is smaller
than the radii of convergence of both power series 
\[
\mathcal{G}\left(z\right)=\frac{1}{1-\mathfrak{F}\left(z\right)}
\]
as long as $\mathfrak{F}\left(z\right)\neq1$. 

\end{thm}
\begin{rem}
\label{rem:some_remarks_general_case}$ $
\begin{enumerate}
\item The radii of convergence of the above power series are at least $\frac{1}{d}$
since the definition of SBRW guarantees that 
\[
\left|\mathscr{E}_{n}\left(\sigma,\sigma\right)\right|,\left|\mathscr{F}_{n}\left(\sigma,\sigma\right)\right|\leq d^{n},\quad\forall n\geq0.
\]

\item This is a high-dimensional analogue of (\ref{eq:return_prob_RW}).
\end{enumerate}
\end{rem}
The following lemma contains several connections between $\left(\mathscr{E}_{n}\right)_{n\geq0}$
and $\left(\mathscr{F}_{n}\right)_{n\geq0}$ which will be useful
for the proof of Theorem \ref{thm:R_via_S_and_overline=00007BS=00007D}:
\begin{lem}
\label{lem:relations_e_K}The following relations hold for every $\sigma,\sigma'\in X_{\pm}^{d-1}$: 
\begin{enumerate}
\item $\mathscr{E}_{n}\left(\sigma,\sigma'\right)=\sum_{\sigma''\in X_{\pm}^{d-1}}E^{\sigma}\left[N_{1}\left(\sigma''\right)\right]\mathscr{E}_{n-1}\left(\sigma'',\sigma'\right).$
\item $E^{\sigma}\left[K_{n}\left(\sigma'\right)\right]=\sum_{\sigma''\in\left(X^{d-1}\backslash\sigma'\right)_{\pm}}E^{\sigma}\left[N_{1}\left(\sigma''\right)\right]E^{\sigma''}\left[K_{n-1}\left(\sigma'\right)\right]$
for $n\geq2$. In particular this gives: 

\begin{enumerate}
\item $\mathscr{F}_{n}\left(\sigma,\sigma'\right)=\sum_{\sigma''\in\left(X^{d-1}\backslash\sigma'\right)_{\pm}}E^{\sigma}\left[N_{1}\left(\sigma''\right)\right]\mathscr{F}_{n-1}\left(\sigma'',\sigma'\right)$
for $n\geq2$. 
\item $E^{\sigma}\left[K_{n}\left(\sigma'\right)\right]=\sum_{\sigma_{1},\ldots,\sigma_{n-1}\in\left(X^{d-1}\backslash\sigma'\right)_{\pm}}E^{\sigma}\left[N_{1}\left(\sigma_{1}\right)\right]E^{\sigma_{1}}\left[N_{1}\left(\sigma_{2}\right)\right]\ldots E^{\sigma_{n-2}}\left[N_{1}\left(\sigma_{n-1}\right)\right]E^{\sigma_{n-1}}\left[N_{1}\left(\sigma'\right)\right]$
for $n\geq2$, and $E^{\sigma}\left[K_{1}\left(\sigma'\right)\right]=E^{\sigma}\left[N_{1}\left(\sigma'\right)\right]$. 
\end{enumerate}
\item $\mathscr{E}_{n}\left(\sigma,\sigma'\right)=\sum_{k=1}^{n}\mathscr{F}_{k}\left(\sigma,\sigma'\right)\mathscr{E}_{n-k}\left(\sigma',\sigma'\right)$
for $n\geq1$.
\end{enumerate}
\end{lem}
\begin{proof}
$ $
\begin{enumerate}
\item The proof follows by the same argument as in Lemma \ref{lem:comparison_of_models}(1). 
\item Using the Markov property, for every $n\geq2$ 
\begin{align*}
 & E^{\sigma}\left[K_{n}\left(\sigma'\right)\right]\\
= & E^{\sigma}\left[\#\left\{ \xi\in\Psi_{n}\left(\sigma'\right)\,:\,\mathfrak{a}^{k}\xi\notin\Psi_{n-k}\left(\sigma'\right)\cup\Psi_{n-k}\left(\overline{\sigma'}\right)\,\forall1\leq k<n\right\} \right]\\
= & \!\!\!\!\!\!\sum_{\,\,\sigma''\in\left(X^{d-1}\backslash\sigma'\right)_{\pm}}\!\!\!\!\!\!\!\! E^{\sigma}\left[N_{1}\left(\sigma_{1}\right)\right]E^{\sigma''}\left[\#\left\{ \xi\in\Psi_{n-1}\left(\sigma'\right)\,:\,\mathfrak{a}^{k}\xi\notin\Psi_{n-1-k}\left(\sigma'\right)\cup\Psi_{n-1-k}\left(\overline{\sigma'}\right)\,\forall1\leq k<n-1\right\} \right]\\
= & \!\!\!\!\!\!\sum_{\,\,\sigma''\in\left(X^{d-1}\backslash\sigma'\right)_{\pm}}\!\!\!\!\!\!\!\! E^{\sigma}\left[N_{1}\left(\sigma''\right)\right]E^{\sigma''}\left[K_{n-1}\left(\sigma\right)\right].
\end{align*}
(2)(a) follows from the fact that $\mathscr{F}_{n}\left(\sigma,\sigma'\right)=E^{\sigma}\left[F_{n}\left(\sigma'\right)\right]=E^{\sigma}\left[K_{n}\left(\sigma'\right)-K_{n}\left(\overline{\sigma'}\right)\right]$.
(2)(b) follows by induction using the fact that for $n=1$:
\begin{align*}
E^{\sigma}\left[K_{1}\left(\sigma'\right)\right] & =E^{\sigma}\left[\#\left\{ \xi\in\Psi_{1}\left(\sigma'\right)\,:\,\mathfrak{a}^{k}\xi\notin\Psi_{1-k}\left(\sigma'\right)\cup\Psi_{1-k}\left(\overline{\sigma'}\right)\,\forall1\leq k<1\right\} \right]\\
 & =E^{\sigma}\left[\left|\Psi_{1}\left(\sigma'\right)\right|\right]=E^{\sigma}\left[N_{1}\left(\sigma'\right)\right].
\end{align*}

\item The proof follows by induction and the Markov property. First note
that for $n=1$ 
\[
\mathscr{E}_{1}\left(\sigma,\sigma'\right)=E^{\sigma}\left[N_{1}\left(\sigma'\right)\right]=\sum_{k=1}^{1}\mathscr{F}_{k}\left(\sigma,\sigma'\right)\mathscr{E}_{1-k}\left(\sigma',\sigma'\right).
\]
Assume next that the relation holds for $n-1$, then by part (1) 
\begin{align*}
\mathscr{E}_{n}\left(\sigma,\sigma'\right) & =\sum_{\sigma''\in X_{\pm}^{d-1}}E^{\sigma}\left[N_{1}\left(\sigma''\right)\right]\mathscr{E}_{n-1}\left(\sigma'',\sigma'\right)\\
 & =\sum_{\sigma''\in X_{\pm}^{d-1}}E^{\sigma}\left[N_{1}\left(\sigma''\right)\right]\sum_{k=1}^{n-1}\mathscr{F}_{k}\left(\sigma'',\sigma'\right)\mathscr{E}_{n-1-k}\left(\sigma',\sigma'\right)\\
 & =\sum_{k=1}^{n-1}\left(\sum_{\sigma''\in X_{\pm}^{d-1}}E^{\sigma}\left[N_{1}\left(\sigma''\right)\right]\mathscr{F}_{k}\left(\sigma'',\sigma'\right)\right)\mathscr{E}_{n-1-k}\left(\sigma',\sigma'\right).
\end{align*}
However by (2)(a) 
\begin{align*}
\sum_{\sigma''\in X_{\pm}^{d-1}}E^{\sigma}\left[N_{1}\left(\sigma''\right)\right]\mathscr{F}_{k}\left(\sigma'',\sigma'\right) & =\sum_{\sigma''\in\left(X^{d-1}\backslash\sigma'\right)_{\pm}}E^{\sigma}\left[N_{1}\left(\sigma''\right)\right]\mathscr{F}_{k}\left(\sigma'',\sigma'\right)\\
 & +E^{\sigma}\left[N_{1}\left(\sigma'\right)\right]\mathscr{F}_{k}\left(\sigma',\sigma'\right)+E^{\sigma}\left[N_{1}\left(\overline{\sigma'}\right)\right]\mathscr{F}_{k}\left(\overline{\sigma'},\sigma'\right)\\
 & =\mathscr{F}_{k+1}\left(\sigma,\sigma'\right)+\left(E^{\sigma}\left[N_{1}\left(\sigma'\right)\right]-E^{\sigma}\left[N_{1}\left(\overline{\sigma'}\right)\right]\right)\mathscr{F}_{k}\left(\sigma',\sigma'\right)\\
 & =\mathscr{F}_{k+1}\left(\sigma,\sigma'\right)+\mathscr{F}_{1}\left(\sigma,\sigma'\right)\mathscr{F}_{k}\left(\sigma',\sigma'\right)
\end{align*}
and therefore by the induction hypothesis 
\begin{align*}
 & \mathscr{E}_{n}\left(\sigma,\sigma'\right)\\
= & \sum_{k=1}^{n-1}\mathscr{F}_{k+1}\left(\sigma,\sigma'\right)\mathscr{E}_{n-1-k}\left(\sigma',\sigma'\right)+\sum_{k=1}^{n-1}\mathscr{F}_{1}\left(\sigma,\sigma'\right)\mathscr{F}_{k}\left(\sigma',\sigma'\right)\mathscr{E}_{n-1-k}\left(\sigma',\sigma'\right)\\
= & \sum_{k=2}^{n}\mathscr{F}_{k}\left(\sigma,\sigma'\right)\mathscr{E}_{n-k}\left(\sigma',\sigma'\right)+\mathscr{F}_{1}\left(\sigma,\sigma'\right)\cdot\sum_{k=1}^{n-1}\mathscr{F}_{k}\left(\sigma',\sigma'\right)\mathscr{E}_{n-1-k}\left(\sigma',\sigma'\right)\\
= & \sum_{k=2}^{n}\mathscr{F}_{k}\left(\sigma,\sigma'\right)\mathscr{E}_{n-k}\left(\sigma',\sigma'\right)+\mathscr{F}_{1}\left(\sigma,\sigma'\right)\cdot\mathscr{E}_{n-1}\left(\sigma',\sigma'\right)\\
= & \sum_{k=1}^{n}\mathscr{F}_{k}\left(\sigma,\sigma'\right)\mathscr{E}_{n-k}\left(\sigma',\sigma'\right).
\end{align*}

\end{enumerate}
\end{proof}

\begin{proof}[Proof of Proposition \ref{thm:R_via_S_and_overline=00007BS=00007D}]
 The statement will follow once we show that for every $z\in\mathbb{C}$
whose absolute value is smaller than the radii of convergence of both
power series 
\[
\mathcal{G}\left(z\right)=1+\mathfrak{F}\left(z\right)\mathcal{G}\left(z\right).
\]
This however follows from Lemma \ref{lem:relations_e_K}(3) since
\begin{align*}
\mathcal{G}\left(z\right) & =\sum_{n=0}^{\infty}\mathscr{E}_{n}\left(\sigma_{0},\sigma_{0}\right)z^{n}=1+\sum_{n=1}^{\infty}\mathscr{E}_{n}\left(\sigma_{0},\sigma_{0}\right)z^{n}=1+\sum_{n=1}^{\infty}\left(\sum_{k=1}^{n}\mathscr{F}_{k}\left(\sigma_{0},\sigma_{0}\right)\mathscr{E}_{n-k}\left(\sigma_{0},\sigma_{0}\right)\right)z^{n}\\
 & =1+\sum_{k=1}^{\infty}\mathscr{F}_{k}\left(\sigma_{0},\sigma_{0}\right)z^{k}\left(\sum_{n=k}^{\infty}\mathscr{E}_{n-k}\left(\sigma_{0},\sigma_{0}\right)z^{n-k}\right)=1+\left(\sum_{k=1}^{\infty}\mathscr{F}_{k}\left(\sigma_{0},\sigma_{0}\right)z^{k}\right)\cdot\mathcal{G}\left(z\right)\\
 & =1+\mathfrak{F}\left(z\right)\mathcal{G}\left(z\right).
\end{align*}

\end{proof}

\section{Arboreal complexes \label{sec:arboreal_complexes}}

The goal of this Section is study ESBRW on arboreal complexes. The
two main results are Theorem \ref{thm:The_spectral_measure_of_arboreal_complexes},
which is proved in Subsections \ref{sub:Finding_G(z)} and \ref{sub:Finding-the-spectral-measure},
and Corollary \ref{cor:Transience_and_recurrence_of_T^d_k}, which
is proved in Subsection \ref{sub:Transience-and-recurrence_on_arboreal_complexes}.
The proof of Theorem \ref{thm:The_spectral_measure_of_arboreal_complexes}
is separated into two parts. First, using the transitive and tree
like structure of the regular arboreal complex, we find an explicit
formula for $\mathcal{G}\left(z\right)=\sum_{n=0}^{\infty}\mathscr{E}_{n}\left(\sigma,\sigma\right)z^{n}$.
Secondly, using the precise expression for $\mathcal{G}\left(z\right)$
and the Stieltjes transform the spectral measure is obtained. The
proof is similar in spirit to Kesten's proof for $k$-regular trees
\cite{Ke59}, however a special care is needed since the terms $\mathscr{E}_{n}\left(\sigma,\sigma\right)$
and $\mathscr{F}_{n}\left(\sigma,\sigma\right)$ are not non-negative
as in the graph case. 

We start by recalling the definition of arboreal complexes:
\begin{defn}[{Arboreal complexes \cite[Definition 3.2]{PR12}}]
 \label{def:regular_arboreal_complex}We say that a $d$-complex
is \emph{arboreal }if it is $\left(d-1\right)$-connected, and has
no simple $d$-loops. That is, there are no non-backtracking closed
loops of $d$-cells, $\tau_{0},\tau_{1},\ldots,\tau_{n}=\tau_{0}$
such that $\dim\left(\tau_{i}\cap\tau_{i+1}\right)=d-1$ and $\tau_{i}\neq\tau_{i+2}$
(the chain is non-backtracking). As in the graph case for every $k\geq1$
there exists a unique $k$-regular arboreal $d$-complex denoted $T_{k}^{d}$. 
\end{defn}
The following choice of oriented cells in $T_{k}^{d}$ will be useful
for the proof: Choose an arbitrary $\left(d-1\right)$-cell $\sigma_{0}\in X_{\pm}^{d-1}$
and call it the $0^{th}$ layer of $T_{k}^{d}$. Define the $1^{st}$
layer to be all the oriented $\left(d-1\right)$-cells which are neighbors
of $\sigma_{0}$ (there are $k\cdot d$ such cells) and denote one
of them by $\sigma_{1}$. The $2^{nd}$ layer of $T_{k}^{d}$ is the
set of oriented $\left(d-1\right)$-cells which are neighbors of a
$\left(d-1\right)$-cell in the $1^{st}$ layer, such that none of
their oriented versions are in the $0^{th}$ or $1^{st}$ layers.
Finally, let $\sigma_{2}$ be a representative in the $2^{nd}$ layer
which is a neighbor of $\sigma_{1}$. One can continue in the same
manner, defining all the layers of $T_{k}^{d}$ around $\sigma_{0}$,
eventually ending up with a choice of orientation for $T_{k}^{d}$.
Here however, we don't need  the full layer structure. Figure \ref{fig:The-orientation-T_2_2}
demonstrates a choice for $\sigma_{0},\sigma_{1},\sigma_{2}$ and
the layers structure in $T_{2}^{2}$.

\begin{figure}[h]
\begin{centering}
\includegraphics[scale=1.2]{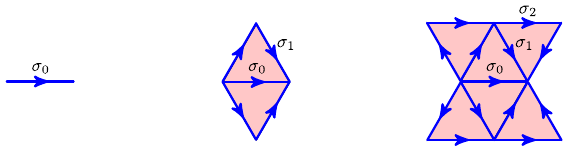}
\par\end{centering}

\protect\caption{\label{fig:The-orientation-T_2_2}The $0^{th}$, $1^{st}$ and $2^{nd}$
layers in $T_{2}^{2}$ with a choice for $\sigma_{0},\sigma_{1}$
and $\sigma_{2}$. }
\end{figure}

\subsection{Finding $\mathcal{G}\left(z\right)$\label{sub:Finding_G(z)}}

Let 
\[
\mathcal{G}\left(z\right)=\sum_{n=0}^{\infty}\mathscr{E}_{n}^{p}\left(\sigma_{0},\sigma_{0}\right)z^{n}\quad,\quad\mathfrak{F}\left(z\right)=\sum_{n=0}^{\infty}\mathscr{F}_{n}^{p}\left(\sigma_{0},\sigma_{0}\right)\quad,\quad\mathcal{U}\left(z\right)=\sum_{n=0}^{\infty}\mathscr{F}_{n}^{p}\left(\sigma_{1},\sigma_{0}\right)z^{n},
\]
and set $\mathbf{r}$ to be the minimum of the radii of convergence
of the above power series%
\footnote{As noted in Remark \ref{rem:some_remarks_general_case}, $\mathbf{r}\geq\frac{1}{d}>0$.%
}. 
\begin{lem}
\label{lem:Relation_between_generation_functions_for_arboreal_complexes}For
the $p$-lazy ESBRW on $T_{k}^{d}$ the following relations hold for
every $z\in\mathbb{C}$ such that $\left|z\right|<\mathbf{r}$. 
\begin{enumerate}
\item $\mathfrak{F}\left(z\right)=pz+\left(1-p\right)dz\cdot\mathcal{U}\left(z\right).$
\item $\mathcal{U}\left(z\right)=pz\mathcal{U}\left(z\right)+\left(1-p\right)\left[\frac{1}{k}z-\frac{d-1}{k}z\mathcal{U}\left(z\right)+\frac{k-1}{k}dz\left(\mathcal{U}\left(z\right)\right)^{2}\right]$.
\end{enumerate}
\end{lem}
\begin{proof}
$ $
\begin{enumerate}
\item By Lemma \ref{lem:relations_e_K}(2) 
\begin{align}
\mathfrak{F}\left(z\right)= & \sum_{n=0}^{\infty}\mathscr{F}_{n}\left(\sigma_{0},\sigma_{0}\right)z^{n}=\mathscr{F}_{1}\left(\sigma_{0},\sigma_{0}\right)z+\sum_{n=2}^{\infty}\mathscr{F}_{n}\left(\sigma_{0},\sigma_{0}\right)z^{n}\nonumber \\
= & pz+\sum_{\sigma''\in\left(X^{d-1}\backslash\sigma_{0}\right)_{\pm}}E^{\sigma_{0}}\left[N_{1}\left(\sigma''\right)\right]z\cdot\left(\sum_{n=2}^{\infty}\mathscr{F}_{n-1}\left(\sigma'',\sigma_{0}\right)z^{n-1}\right).\label{eq:finding_mathfrak=00007BF=00007D}
\end{align}
Due to the tree like structure of $T_{k}^{d}$, for $\sigma''\in\left(X^{d-1}\backslash\sigma_{0}\right)_{\pm}$
we have $E^{\sigma_{0}}\left[N_{1}\left(\sigma''\right)\right]=0$
unless $\sigma''$ is in the $1^{st}$ layer of $T_{k}^{d}$. In addition
by the transitive structure of $T_{k}^{d}$ the power series $\sum_{n=2}^{\infty}\mathscr{F}_{n-1}\left(\sigma'',\sigma_{0}\right)z^{n-1}$
is the same for every $\sigma''$ in the first layer of $T_{k}^{d}$
and equals $\mathcal{U}\left(z\right)$. Thus
\begin{align*}
\mathfrak{F}\left(z\right) & =pz+\sum_{\tiny{\begin{array}{c}
\sigma''\in\mbox{first }\\
\mbox{layer of }X_{+}^{d-1}
\end{array}}}E^{\sigma_{0}}\left[N_{1}\left(\sigma''\right)\right]z\cdot\mathcal{U}\left(z\right)=pz+\left(1-p\right)z\cdot\mathcal{U}\left(z\right).
\end{align*}

\item As in part (1) the claim follows by a one step analysis of the ESBRW.
Using the Markov property, Lemma \ref{lem:relations_e_K}(2) and a
similar argument to the one in (\ref{eq:finding_mathfrak=00007BF=00007D})
\begin{align*}
\mathcal{U}\left(z\right) & =\mathscr{F}_{1}\left(\sigma_{1},\sigma_{0}\right)z+\sum_{\sigma''\in X_{\pm}^{d-1}}E^{\sigma_{1}}\left[N_{1}\left(\sigma''\right)\right]z\cdot\left(\sum_{n=2}^{\infty}\mathscr{F}_{n-1}\left(\sigma'',\sigma_{0}\right)z^{n-1}\right).
\end{align*}
Due to the tree structure of $T_{k}^{d}$ 
\[
E^{\sigma_{1}}\left[N_{1}\left(\sigma''\right)\right]=\begin{cases}
p & \quad,\,\sigma''=\sigma_{1}\\
\left(1-p\right)\frac{1}{k} & \quad,\,\sigma''=\sigma_{0}\\
\left(1-p\right)\frac{1}{k} & \quad,\,\sigma''\,\mbox{is in the }\ensuremath{2^{nd}}\,\mbox{layer of }T_{k}^{d}\,\mbox{and }\sigma''\sim\sigma_{1}\\
\left(1-p\right)\frac{1}{k} & \quad,\,\overline{\sigma''}\,\mbox{is in the }\ensuremath{1^{st}}\,\mbox{layer of }T_{k}^{d}\,\mbox{and }\sigma''\cup\sigma_{1}\,\mbox{is a }d-\mbox{cell}\\
0 & \quad,\,\mbox{otherwise}
\end{cases}
\]
and by its transience 
\[
\sum_{n=2}^{\infty}\mathscr{F}_{n-1}\left(\sigma'',\sigma_{0}\right)z^{n-1}=\begin{cases}
\mathcal{U}\left(z\right) & ,\,\sigma''=\sigma_{1}\\
\sum_{n=1}^{\infty}\limits\mathscr{F}_{n}\left(\sigma_{2},\sigma_{0}\right)z^{n} & ,\,\sigma''\,\mbox{is in the }\ensuremath{2^{nd}}\,\mbox{layer of }T_{k}^{d}\,\mbox{and }\sigma''\sim\sigma_{1}\\
-\mathcal{U}\left(z\right) & ,\,\overline{\sigma''}\,\mbox{is in the }\ensuremath{1^{st}}\,\mbox{layer of }T_{k}^{d}\,\mbox{and }\sigma''\cup\sigma_{1}\\
 & \,\,\mbox{is a }d-\mbox{cell}
\end{cases}.
\]
Finally, note that the number of $\sigma''$ in the $2^{nd}$ layer
of $T_{k}^{d}$ such that $\sigma''\sim\sigma_{1}$ is exactly $d\left(k-1\right)$
and that the number of $\sigma''$such that $\overline{\sigma''}$
is in the $1^{st}$ layer and $\sigma''\cup\sigma_{1}$ is a $d-$cell
is exactly $d-1$. Combining all of the above gives 
\[
\mathcal{U}\left(z\right)=pz\cdot\mathcal{U}\left(z\right)+\left(1-p\right)\left[\frac{1}{k}z-\frac{d-1}{k}z\cdot\mathcal{U}\left(z\right)+\frac{k-1}{k}dz\cdot\sum_{n=1}^{\infty}\mathscr{F}_{n}\left(\sigma_{2},\sigma_{0}\right)z^{n}\right].
\]
Thus, the proof will be complete once we show that $\sum_{n=1}^{\infty}\mathscr{F}\left(\sigma_{2},\sigma_{0}\right)z^{n}=\left(\mathcal{U}\left(z\right)\right)^{2}$.
Since each particle starting in $\sigma_{2}$ must split through either
$\sigma_{1}$ or $\overline{\sigma_{1}}$ in order to reach $\sigma_{0}$
we can rewrite $\mathscr{F}_{n}\left(\sigma_{2},\sigma_{0}\right)$
as a sum according to the first ``visit'' to one of these cells.
This gives
\begin{align*}
 & \sum_{n=1}^{\infty}\mathscr{F}_{n}\left(\sigma_{2},\sigma_{0}\right)z^{n}=\sum_{n=0}^{\infty}E^{\sigma_{2}}\left[F_{n}\left(\sigma_{0}\right)\right]z^{n}\\
= & \sum_{n=0}^{\infty}\sum_{k=0}^{n}\left[E^{\sigma_{2}}\left[K_{k}\left(\sigma_{1}\right)\right]E^{\sigma_{1}}\left[F_{n-k}\left(\sigma_{0}\right)\right]+E^{\sigma_{2}}\left[K_{k}\left(\overline{\sigma_{1}}\right)\right]E^{\overline{\sigma_{1}}}\left[F_{n-k}\left(\sigma_{0}\right)\right]\right]z^{n}\\
\overset{_{\left(\star\right)}}{=} & \sum_{n=0}^{\infty}\sum_{k=0}^{n}\left[E^{\sigma_{2}}\left[K_{k}\left(\sigma_{1}\right)\right]E^{\sigma_{1}}\left[F_{n-k}\left(\sigma_{0}\right)\right]-E^{\sigma_{2}}\left[K_{k}\left(\overline{\sigma_{1}}\right)\right]E^{\sigma_{1}}\left[F_{n-k}\left(\sigma_{0}\right)\right]\right]z^{n}\\
= & \sum_{n=0}^{\infty}\sum_{k=0}^{n}E^{\sigma_{2}}\left[F_{k}\left(\sigma_{1}\right)\right]E^{\sigma_{1}}\left[F_{n-k}\left(\sigma_{2}\right)\right]z^{n}=\sum_{n=0}^{\infty}\sum_{k=0}^{n}\mathscr{F}_{k}\left(\sigma_{2},\sigma_{1}\right)z^{k}\cdot\mathscr{F}_{n-k}\left(\sigma_{1},\sigma_{0}\right)z^{n-k}\\
= & \left(\mathcal{U}\left(z\right)\right)^{2},
\end{align*}
where for $\left(\star\right)$ we used the fact that $E^{\overline{\sigma}}\left[K_{k}\left(\sigma'\right)\right]=E^{\sigma}\left[K_{k}\left(\overline{\sigma'}\right)\right]$
for every $k\geq0$ and $\sigma,\sigma'\in\left(T_{k}^{d}\right)_{\pm}^{d-1}$.

\end{enumerate}
\end{proof}
\medskip{}

Using Lemma \ref{lem:Relation_between_generation_functions_for_arboreal_complexes}
we can now find $\mathcal{G}\left(z\right).$ For simplicity fix $p=0$
and note that in this case Lemma \ref{lem:Relation_between_generation_functions_for_arboreal_complexes}(2)
gives $\mathcal{U}\left(z\right)=\frac{1}{k}z-\frac{d-1}{k}z\mathcal{U}\left(z\right)+\frac{k-1}{k}dz\left(\mathcal{U}\left(z\right)\right)^{2}$.
The solutions of the equation are 

\[
L_{\pm}\left(z\right)=\frac{\left(d-1\right)z+k\pm\sqrt{\left(\left(d-1\right)z+k\right)^{2}-4\left(k-1\right)dz^{2}}}{2\left(k-1\right)dz},
\]
and since only the solution $L_{-}$ satisfies $L_{-}\left(0\right)=0=\mathcal{U}\left(0\right)$
we conclude that $\mathcal{U}=L_{-}$. Using Lemma \ref{lem:Relation_between_generation_functions_for_arboreal_complexes}(1)
and Proposition \ref{thm:R_via_S_and_overline=00007BS=00007D} it
follows that as long as $\mathfrak{F}\left(z\right)\neq1$
\[
\mathcal{G}\left(z\right)=\frac{1}{1-\mathfrak{F}\left(z\right)}=\frac{1}{1-dz\mathcal{U}\left(z\right)}
\]
which gives 
\[
\mathcal{G}\left(z\right)=\frac{2\left(k-1\right)}{k-2-\left(d-1\right)z+\sqrt{\left(\left(d-1\right)z+k\right)^{2}-4\left(k-1\right)dz^{2}}}.
\]
Note that the singularity points of $\mathcal{G}$ are the points
where the denominator is zero (which are in fact the points where
$\mathfrak{F}\left(z\right)=1$) and the points where the square-root
is zero. Those are given by 
\[
z=1,\,\mbox{when }k\leq d+1
\]
 and 
\[
z_{\pm}=\frac{k}{1-d\mp2\sqrt{\left(k-1\right)d}}.
\]
In particular we infer that 
\[
\left(\limsup_{n\to\infty}\sqrt[n]{\left|\mathscr{E}_{n}^{0}\left(\sigma_{0},\sigma_{0}\right)\right|}\right)^{-1}=\left(\begin{array}{c}
\mbox{radius of }\\
\mbox{convergence of }\mathcal{G}
\end{array}\right)=\begin{cases}
\min\left\{ 1,\left|\frac{k}{d-1+2\sqrt{\left(k-1\right)d}}\right|\right\}  & k\leq d+1\\
\frac{k}{d-1+2\sqrt{\left(k-1\right)d}} & k>d+1
\end{cases}.
\]

\subsection{Finding the spectral measure \label{sub:Finding-the-spectral-measure}}

Once the moment generating function $\mathcal{G}\left(z\right)$ is
known the spectral measure can be calculated using the Stieltjes transform.
Let $\mu_{d,k}$ be the spectral measure associated with the operator
$\mathscr{A}_{0}$ of the arboreal complex $T_{k}^{d}$ and for $z\in\mathbb{C}\backslash\mathbb{R}$
let $\mathcal{S}\left(z\right)=\int_{\mathbb{R}}\frac{1}{x-z}d\mu_{d,k}\left(x\right)$
be its Stieltjes transform. Note that for $z\in\mathbb{C}\backslash\mathbb{R}$
whose absolute value is bigger than $\max\left\{ \left|\lambda\right|\,:\,\lambda\in\mbox{support of}\,\mu_{d,k}\right\} $
\begin{align*}
\mathcal{S}\left(z\right) & =\int_{\mathbb{R}}\frac{1}{x-z}d\mu_{d,k}\left(x\right)=-\frac{1}{z}\int_{\mathbb{R}}\frac{1}{1-\frac{x}{z}}d\mu_{d,k}\left(x\right)\\
 & =-\frac{1}{z}\int_{\mathbb{R}}\sum_{n=0}^{\infty}\left(\frac{x}{z}\right)^{n}d\mu_{d,k}\left(x\right)=-\frac{1}{z}\sum_{n=0}^{\infty}\mathscr{E}_{n}\left(\sigma_{0},\sigma_{0}\right)\frac{1}{z^{n}}=-\frac{1}{z}\mathcal{G}\left(\frac{1}{z}\right).
\end{align*}
Since $\mathcal{S}\left(z\right)$ and $-\frac{1}{z}\mathcal{G}\left(\frac{1}{z}\right)$
agree on an open ball it follows that their analytic continuations%
\footnote{One can avoid the use of analytic continuation by working with the
$p$-lazy SBRW for any $p>\frac{d-1}{d+1}$. %
} agree and in particular that

\[
\mathcal{S}\left(z\right)=-\frac{2\left(k-1\right)}{\left(k-2\right)z-\left(d-1\right)+\sqrt{\left(d-1+kz\right)^{2}-4\left(k-1\right)d}}.
\]
Having found the Stieltjes transform $\mathcal{S}$ we turn to evaluate
$\mu_{d,k}$ starting with the spectral density $\rho_{d,k}$, namely,
the Radon-Nikodym derivative of the absolutely continuous part. This
is done by evaluating the limit 
\[
\lim_{\varepsilon\downarrow0}\frac{1}{\pi}\int_{\mathbb{R}}\frac{\varepsilon}{\left(x-x_{0}\right)^{2}+\varepsilon^{2}}d\mu_{d,k}\left(x\right)
\]
which by the dominated convergence theorem equals $\rho_{d,k}\left(x_{0}\right)$
when $\mu_{d,k}$ doesn't have an atom in $x_{0}$ and $+\infty$
when it does. For every $x_{0}\in\mathbb{R}$ (except for $x_{0}=1$
when $k\leq d+1$)

\begin{align}
\lim_{\varepsilon\downarrow0}\frac{1}{\pi}\int_{\mathbb{R}}\frac{\varepsilon}{\left(x-x_{0}\right)^{2}+\varepsilon^{2}}d\mu_{d,k}\left(x\right) & =\lim_{\varepsilon\downarrow0}\frac{1}{\pi}\mathrm{Im}\left(\mathcal{S}\left(x_{0}+i\varepsilon\right)\right)\nonumber \\
 & =-\frac{1}{\pi}\mathrm{Im}\left(\frac{2\left(k-1\right)}{\left(k-2\right)x_{0}-\left(d-1\right)+\sqrt{\left(d-1+kx_{0}\right)^{2}-4\left(k-1\right)d}}\right).\label{eq:Stieltjes_trans_limit}
\end{align}
The right hand side of (\ref{eq:Stieltjes_trans_limit}) equals zero
whenever $\left(d-1+kx_{0}\right)^{2}\geq4\left(k-1\right)d$ and
\[
\frac{\sqrt{4\left(k-1\right)d-\left(d-1+kx\right)^{2}}}{2\pi\left(d+x\right)\left(1-x\right)}
\]
when $\left(d-1+kx_{0}\right)^{2}\leq4\left(k-1\right)d$. Since $\left(d-1+kx_{0}\right)^{2}\leq4\left(k-1\right)d$
exactly when 
\[
x_{0}\in I_{d,k}\equiv\left[\frac{1-d-2\sqrt{\left(k-1\right)d}}{k},\frac{1-d+2\sqrt{\left(k-1\right)d}}{k}\right],
\]
it follows that the density function is 
\[
\rho_{d,k}\left(x\right)=\frac{\sqrt{4\left(k-1\right)d-\left(d-1+kx\right)^{2}}}{2\pi\left(d+x\right)\left(1-x\right)}\chi_{x\in I_{d,k}}.
\]
One can now verify that 
\[
\int_{I_{d,k}}\rho_{d,k}\left(x\right)dx=\begin{cases}
\frac{k}{d+1} & \,\,,k<d+1\\
1 & \,\,,k\geq d+1
\end{cases},
\]
which suggest that the size of the atom in the unique suspect for
being one, i.e., $x_{0}=1$ when $k\leq d+1$, is $\frac{d+1-k}{d+1}$.
A direct proof of this fact without calculating the above integral
can also be given using $\mathcal{S}$. Define

\textbf{
\[
h\left(\varepsilon\right):=-i\int_{\mathbb{R}}\frac{\varepsilon}{x-1-i\varepsilon}d\mu_{d,k}\left(x\right)=-i\varepsilon\mathcal{S}\left(1+i\varepsilon\right).
\]
}By the dominated convergence theorem\textbf{ $\lim_{\varepsilon\downarrow0}h\left(\varepsilon\right)=\mu_{d,k}\left(\left\{ 1\right\} \right)$}
and therefore\textbf{ 
\begin{align*}
\mu_{k,d}\left(\left\{ 1\right\} \right) & =\lim_{\varepsilon\downarrow0}-i\varepsilon\mathcal{S}\left(1+i\varepsilon\right)=\begin{cases}
\frac{d+1-k}{d+1} & \quad,\, k\leq d+1\\
0 & \quad,\, k=d+1
\end{cases}.
\end{align*}
}This shows that 
\[
\mu_{d,k}\left(A\right)=\begin{cases}
\int_{A}\rho_{d,k}\left(x\right)dx+\frac{d+1-k}{d+1}\chi_{0\in A} & ,\, k<d+1\\
\int_{A}\rho_{d,k}\left(x\right)dx & ,\, k\geq d+1
\end{cases}
\]
and completes the proof. \hfill{}$\boxempty$

Let us take this opportunity to state a conjecture regarding the eigenvalue
$1$ in simplicial complexes. Theorem \ref{thm:The_spectral_measure_of_arboreal_complexes}
implies in particular that $1$ is an eigenvalue of $\mathscr{A}_{0}$
in $T_{k}^{d}$ as long as $k\leq d$. We conjecture that this holds
in a much bigger generality:
\begin{conjecture}
One is an eigenvalue of $\mathscr{A}_{0}$ for every $d$-complex
$X$ such that $\sup_{\sigma\in X^{d-1}}\deg\left(\sigma\right)\leq d$. 
\end{conjecture}
A weak version of the conjecture is:
\begin{conjecture}[Weaker version]
 One is an eigenvalue of $\mathscr{A}_{0}$ for every arboreal $d$-complex
$X$ such that $\sup_{\sigma\in X^{d-1}}\deg\left(\sigma\right)\leq d$. 
\end{conjecture}

\subsection{Transience and recurrence of ESBRW on regular arboreal complexes\label{sub:Transience-and-recurrence_on_arboreal_complexes}}

The notion of transient $\left(d-1\right)$-walk was defined in \cite[Subsection 3.8]{PR12}.
A slightly more general analogue for the ESBRW is:
\begin{defn}
\label{def:Transience_and_recurrence}The ESBRW is called transient
if $\sum_{n=0}^{\infty}\mathscr{E}_{n}^{p}\left(\sigma,\sigma\right)<\infty$
for every $\sigma\in X^{d-1}$ and some $\frac{d-1}{d+1}<p<1$. If
$\sum_{n=0}^{\infty}\mathscr{E}_{n}^{p}\left(\sigma,\sigma\right)<\infty$
for some $\sigma\in X^{d-1}$ and $\frac{d-1}{d+1}<p<1$ the random
walk is called recurrent.
\end{defn}
Let $X$ be a $d$-complex and denote by $\mu^{p}$ the spectral measure
of $\mathscr{A}_{p}$ associated with the function $\ind_{\sigma}$.
Since $\mathscr{A}_{p}=pI+\left(1-p\right)\mathscr{A}_{0}$ , it follows
that $\int_{\mathbb{R}}f\left(x\right)d\mu^{p}\left(x\right)=\int_{\mathbb{R}}f\left(p+\left(1-p\right)x\right)d\mu^{0}\left(x\right)$
for every integrable function $f:\mathbb{R}\to\mathbb{R}$. In addition,
since $\mathrm{Support}\left(\mu^{p}\right)\subset\mathrm{Spec}\left(\mathscr{A}_{p}\right)\subset\left[1-\left(1-p\right)\left(d+1\right),1\right]$,
it follows that the support of the measure $\mu^{p}$ is contained
in $\left(-1,1\right]$ for every $\frac{d-1}{d+1}<p<1$. Therefore,
by the monotone convergence theorem and the relation between $\mu^{0}$
and $\mu^{p}$ 
\begin{align}
\sum_{n=0}^{\infty}\mathscr{E}_{n}^{p}\left(\sigma,\sigma\right) & =\sum_{n=0}^{\infty}\deg\left(\sigma\right)\cdot\left\langle \mathscr{A}_{p}^{n}\ind_{\sigma},\ind_{\sigma}\right\rangle =\deg\left(\sigma\right)\cdot\sum_{n=0}^{\infty}\int_{\mathbb{R}}x^{n}d\mu^{p}\left(x\right)\nonumber \\
 & =\deg\left(\sigma\right)\cdot\int_{\mathbb{R}}\frac{1}{1-x}d\mu^{p}\left(x\right)=\deg\left(\sigma\right)\cdot\int_{\mathbb{R}}\frac{1}{1-\left(p+\left(1-p\right)x\right)}d\mu^{0}\left(x\right)\nonumber \\
 & =\frac{1}{1-p}\cdot\int_{\mathbb{R}}\frac{1}{1-x}d\mu^{0}\left(x\right).\label{eq:tran_recu_1}
\end{align}
In particular the $p$-lazy branching random walk is recurrent/transient
for some $\frac{d-1}{d+1}<p<1$ if and only if it is recurrent/transient
for every such $p$.

\begin{proof}[Proof of Corollary \ref{cor:Transience_and_recurrence_of_T^d_k}]
 By the above argument, in order to check recurrence/transience of
the ESBRW it suffices to check whether the integral $\int_{\mathbb{R}}\frac{1}{1-x}d\mu_{d,k}\left(x\right)$
is infinite/finite respectively. When $k\leq d$, $1$ is an atom
of the measure $\mu_{d,k}$ and therefore the integral is infinite.
If $k>d+1$ the spectrum of $\mu_{d,k}$ is a compact subset of $\left(-\infty,1\right)$
and therefore the integral is finite. Finally, in the case $k=d+1$
\begin{align*}
\int_{\mathbb{R}}\frac{1}{1-x}d\mu_{d,k}\left(x\right) & =\int_{I_{d,k}}\frac{1}{1-x}\cdot\frac{\sqrt{4d^{2}-\left(\left(d+1\right)x+\left(d-1\right)\right)^{2}}}{2\pi\left(d+x\right)\left(1-x\right)}dx\\
 & =\int_{\frac{1-3d}{d+1}}^{1}\frac{\sqrt{\left(d+1\right)\left(\left(d+1\right)x+\left(3d-1\right)\right)}}{2\pi\left(d+x\right)}\frac{1}{\left(1-x\right)^{\frac{3}{2}}}dx=\infty,
\end{align*}
which implies that the ESBRW on $T_{d+1}^{d}$ is recurrent.
\end{proof}

\section{Dirichlet problem on simplicial complexes\label{sec:Dirichlet-problem}}

Dirichlet problem concerns with finding a function that solves a partial
differential equation (PDE) with prescribed boundary values. The PDE
which is usually under consideration is Laplace's equation. 

In the discrete setting of graphs Dirichlet problem is stated as follows:\\
\\
\fbox{\begin{minipage}[t]{1\columnwidth}%
\textbf{Discrete Dirichlet problem:} Given a finite graph $G=\left(V,E\right)$,
a non-empty subset $A\subset V$ and a function $f:A\to\mathbb{R}$
find a solution $F:V\to\mathbb{R}$ to the boundary value problem
\[
\begin{cases}
\Delta^{+}F\left(x\right) & ,\forall x\in V\backslash A\\
F\left(x\right)=f\left(x\right) & ,\forall x\in A
\end{cases}.
\]
\end{minipage}}\\

If $G$ is a connected graph, then for every non-empty set $A\subset V$
and $f:A\to\mathbb{R}$ there exists a solution given by $F\left(x\right)=E^{x}\left[f\left(Y_{\tau_{A}}\right)\right],$
where $\left(Y_{n}\right)_{n\geq0}$ is the simple random walk on
the graph $G$ and $\tau_{A}=\inf\left\{ k\geq0\,:\, Y_{k}\in A\right\} $.
In addition, the solution is unique due to the maximum principle.

\medskip{}

A high-dimensional counterpart of the problem for forms is:\\
\\
\fbox{\begin{minipage}[t]{1\columnwidth}%
\textbf{High-dimensional discrete Dirichlet problem: }Given a finite
$d$-complex $X$, a non-empty subset $A\subset X^{d-1}$ and a form
$f:A_{\pm}\to\mathbb{R}$ (where $A_{\pm}$ is the set of oriented
$\left(d-1\right)$-cells whose unoriented version is in $A$) find
a solution $F\in\Omega^{d-1}$ to the boundary value problem
\[
\begin{cases}
\Delta^{+}F\left(\sigma\right) & ,\forall\sigma\in\left(X^{d-1}\backslash A\right)_{\pm}\\
F\left(\sigma\right)=f\left(\sigma\right) & ,\forall\sigma\in A_{\pm}
\end{cases}.
\]
\end{minipage}}\\

The situation in high dimensions is more involved and for a general
set $A$ one can have infinitely many solutions. For example if $X$
is composed of a single triangle $t=\left\{ v_{0},v_{1},v_{2}\right\} $,
$A=\left\{ \left\{ v_{0},v_{1}\right\} \right\} $ and $f\left(\left[v_{0},v_{1}\right]\right)=-f\left(\left[v_{1},v_{0}\right]\right)=1$,
the form defined by 
\[
F_{\alpha}\left(e\right)=\begin{cases}
1 & \quad,\, e=\left[v_{0},v_{1}\right]\\
\alpha & \quad,\, e=\left[v_{1},v_{2}\right]\\
-1-\alpha & \quad,\, e=\left[v_{2},v_{0}\right]
\end{cases}
\]
is a solution to the Dirichlet problem for every $\alpha\in\mathbb{R}$. 

Before turning to discuss the existence and uniqueness of solutions
to the high-dimensional Dirichlet problem some additional definitions
are required. Let $X$ be a $d$-complex and $\emptyset\neq A\subset X^{d-1}$.
Since the case $A=X^{d-1}$ is degenerate and has exactly one solution,
$F=f$, we assume without loss of generality that $A\neq X^{d-1}$. 

Consider $\Delta^{+}$ as a matrix and denote by $\Delta_{X\backslash A}^{+}$
its restriction to rows and columns of $\left(d-1\right)$-cells in
$X^{d-1}\backslash A$. Similarly let $\delta_{d}^{X\backslash A}$
be the restriction of $\delta_{d}$ to $\left(d-1\right)$-cells in
$X^{d-1}\backslash A$. 

Define the $A$-absorbing, $p$-lazy SBRW on $X$ to be the usual
SBRW except that any particle in $A_{\pm}$ stays put with probability
one. Let $P_{A}$, $E_{A}$ denote the probability and expectation
of the $A$-absorbing SBRW respectively. We can now define the related
effective process $\left(D_{n}\right)_{n\geq0}$ and its Green function
\[
\mathcal{G}_{A}^{p}\left(\sigma,\sigma'\right)=\sum_{n=0}^{\infty}E_{A}^{\sigma}\left[D_{n}\left(\sigma\right)\right],\quad\forall\sigma,\sigma'\in\left(X^{d-1}\backslash A\right)_{\pm}.
\]

Our goal is to prove the following Theorem:
\begin{thm}[Solution to the high-dimensional Dirichlet problem]
\label{Theorem:Dirichlet_probem}Let $X$ be a finite $d$-complex,
$\frac{d-1}{d+1}<p<1$, $\emptyset\neq A\subsetneq X^{d-1}$ such
that $\Delta_{X\backslash A}^{+}$ is invertible and $f:A_{\pm}\to\mathbb{R}$.
Then the unique solution to the Dirichlet problem related to the triplet
$\left(X,A,f\right)$ is the function $F:X_{\pm}^{d-1}\to\mathbb{R}$
given by 
\[
F\left(\sigma\right)=\frac{1}{1-p}\sum_{\sigma''\in A_{\pm}}\left(\sum_{\tiny{\begin{array}{c}
\sigma'\in\left(X^{d-1}\backslash A\right)_{\pm}\\
\sigma'\sim\sigma''
\end{array}}}\frac{\mathcal{G}_{A}^{p}\left(\sigma,\sigma'\right)}{\deg\left(\sigma'\right)}\right)f\left(\sigma''\right).
\]
\end{thm}
\begin{proof}
Decompose%
\footnote{The decomposition here is according to whether the $\left(d-1\right)$-cells
are in $A$ or in $\left(X^{d-1}\backslash A\right)$.%
} the matrix representation of $\Delta^{+}$ as $\Delta^{+}=\left(\begin{array}{c|c}
\Delta_{A}^{+} & -R\\
\hline -Q & \Delta_{X\backslash A}^{+}
\end{array}\right)$. Then $F$ is a solution to the Dirichlet problem if and only if
\[
\left(\begin{array}{c|c}
I & 0\\
\hline -Q & \Delta_{X\backslash A}^{+}
\end{array}\right)\left(\begin{array}{c}
\vdots\\
F\\
\vdots
\end{array}\right)=\left(\begin{array}{c}
f\\
0
\end{array}\right).
\]
Note that this operator is invertible if and only if the operator
$\Delta_{X\backslash A}^{+}$ is invertible, and since the invertibility
of $\Delta_{X\backslash A}^{+}$ was assumed it follows that there
exists a unique solution to Dirichlet problem given by $F|_{X^{d-1}\backslash A}=\left(\Delta_{X\backslash A}^{+}\right)^{-1}Qf$.
Next we show that whenever $\Delta_{X\backslash A}^{+}$ is invertible
and $\frac{d-1}{d+1}<p<1$ its inverse is given by $\frac{1}{1-p}\mathcal{G}_{A}^{p}$
and in particular that $\mathcal{G}_{A}^{p}$ is well defined. Indeed,
by the same argument as in \cite[Proposition 2.7(2)]{PR12} the spectrum
of $\Delta_{X\backslash A}^{+}$ is always a subset of $\left[0,d+1\right]$
and due to the fact that $\Delta_{X\backslash A}^{+}$ is invertible
$\mathrm{Spec}\left(\Delta_{X\backslash A}^{+}\right)\subset\left[\lambda_{0},d+1\right]$
for some $\lambda_{0}>0$. Defining the operator $\mathscr{A}_{p}^{X\backslash A}:=I-\left(1-p\right)\Delta_{X\backslash A}^{+}$,
it follows that $\mathrm{Spec}\left(\mathscr{A}_{p}^{X\backslash A}\right)\subset\left[1-\left(1-p\right)\left(d+1\right),1-\left(1-p\right)\lambda_{0}\right]$.
Since $\frac{d-1}{d+1}<p<1$ this is a closed sub-interval of $\left(-1,1\right)$
and so $\left\Vert \mathscr{A}_{p}^{X\backslash A}\right\Vert =\sup\left\{ \left|\lambda\right|\,:\,\lambda\in\mathrm{Spec}\left(\mathscr{A}_{p}\right)\right\} <1$.
Noting that $\mathscr{A}_{p}^{X\backslash A}$ is the ``transition''
operator of the $A$-absorbing ESBRW for $\left(d-1\right)$-cells
in $X^{d-1}\backslash A$ it follows that $E_{A}^{\sigma}\left[D_{n}\left(\sigma'\right)\right]=\deg\left(\sigma'\right)\cdot\left\langle \left(\mathscr{A}_{p}^{X\backslash A}\right)^{n}\ind_{\sigma},\ind_{\sigma'}\right\rangle $
for every $\sigma,\sigma'\in\left(X^{d-1}\backslash A\right)_{\pm}$.
Thus we conclude that 
\begin{align*}
\left|\mathcal{G}_{A}^{p}\left(\sigma,\sigma'\right)\right| & =\left|\sum_{n=0}^{\infty}E_{A}^{\sigma}\left[D_{n}\left(\sigma'\right)\right]\right|=\deg\left(\sigma'\right)\left|\sum_{n=0}^{\infty}\left\langle \left(\mathscr{A}_{p}^{X\backslash A}\right)^{n}\ind_{\sigma},\ind_{\sigma'}\right\rangle \right|\\
 & \leq\deg\left(\sigma'\right)\sum_{n=0}^{\infty}\left\Vert \mathscr{A}_{p}^{X\backslash A}\right\Vert ^{n}<\infty\quad,\,\forall\sigma,\sigma'\in\left(X^{d-1}\backslash A\right)_{\pm},
\end{align*}
which in particular shows that $\mathcal{G}_{A}^{p}$ is well defined. 

The fact that $\frac{1}{1-p}\mathcal{G}_{A}^{p}=\left(\Delta_{X\backslash A}^{+}\right)^{-1}$
follows now from the Markov property. Indeed,  
\begin{align*}
\mathcal{G}_{A}^{p}\left(\sigma,\sigma'\right) & =\sum_{n=0}^{\infty}E_{A}^{\sigma}\left[D_{n}\left(\sigma'\right)\right]=\ind_{\sigma}\left(\sigma'\right)+\sum_{n=1}^{\infty}\sum_{\sigma''\in\left(X^{d-1}\backslash A\right)_{\pm}}E_{A}^{\sigma}\left[N_{1}\left(\sigma''\right)\right]E_{A}^{\sigma''}\left[D_{n-1}\left(\sigma'\right)\right]\\
 & =\ind_{\sigma}\left(\sigma'\right)+\left(\mathscr{A}_{p}^{X\backslash A}\mathcal{G}_{A}^{p}\right)\left(\sigma,\sigma'\right)
\end{align*}
which gives $\frac{1}{1-p}\mathcal{G}_{A}^{p}\Delta_{X\backslash A}^{+}=\mathcal{G}_{A}^{p}\left(I-\mathscr{A}_{p}^{X\backslash A}\right)=I$. 

Finally, note that $Q$ is nothing else than the restriction of $-\Delta^{+}=\frac{1}{1-p}\left(\mathscr{A}_{p}-I\right)$
to columns of $\left(d-1\right)$-cells in $A$ and rows of $\left(d-1\right)$-cells
in $X^{d-1}\backslash A$. Since there are no diagonal elements in
the restriction this is the same as the restriction of $\frac{1}{1-p}\mathscr{A}_{p}$
to the same rows and columns which for $\sigma'\in\left(X^{d-1}\backslash A\right)_{\pm}$
and $\sigma''\in A_{\pm}$ equals $\frac{1}{1-p}E^{\sigma'}\left[N_{1}\left(\sigma''\right)\right]=\begin{cases}
\frac{1}{\deg\left(\sigma'\right)} & ,\,\sigma''\sim\sigma'\\
0 & ,\,\mbox{otherwise}
\end{cases}$. 
\end{proof}
Before turning to the next section we wish to discuss the main condition
in Theorem \ref{Theorem:Dirichlet_probem}, namely, the invertibility
of $\Delta_{X\backslash A}^{+}$. We start with some simple observations:
\begin{claim}[Invertibility of $\Delta_{X\backslash A}^{+}$]
\label{claim:trivial_invertability_equivalent_conditions} Let $X$
be a finite $d$-complex and $\emptyset\neq A\subsetneq X^{d-1}$.
The following are equivalent:
\begin{enumerate}
\item $\Delta_{X\backslash A}^{+}$ is invertible. 
\item $\ker\delta_{d}^{X\backslash A}=\ker\Delta_{X\backslash A}^{+}$ is
trivial. 
\item For every form $f:\left(X^{d-1}\backslash A\right)_{\pm}\to\mathbb{R}$
which is not identically zero, the extension $\tilde{f}:X_{\pm}^{d-1}\to\mathbb{R}$
given by $\tilde{f}\left(\sigma\right)=\begin{cases}
f\left(\sigma\right) & \sigma\in\left(X^{d-1}\backslash A\right)_{\pm}\\
0 & \sigma\in A_{\pm}
\end{cases}$ is not in $\ker\delta_{d}=\ker\Delta^{+}=Z^{d-1}$.
\item The relative homology $H_{d}\left(X,A\right)$ (see \cite[Section 2.1]{Ha02}
for the definition) is trivial.
\end{enumerate}
\end{claim}
Using the above equivalent definitions we can identify some cases
in which it is easier to check whether $\Delta_{X\backslash A}^{+}$
is invertible or not. Let us start with two definitions:
\begin{defn}
Let $X$ be a finite $d$-complex and $\emptyset\neq A\subsetneq X^{d-1}$.
The set $A$ is called \emph{exhaustive} for the complex $X$ if there
exists a finite sequence $A=A_{0}\subsetneq A_{1}\subsetneq A_{2}\subsetneq\ldots\subsetneq A_{N}=X^{d-1}$
such that for every $n\geq1$ and $\sigma\in A_{n},$ one can find
$\tau\in\mathrm{cf}\left(\sigma\right)$ for which $\mathrm{face}\left(\tau\right)\backslash\sigma\subset A_{n-1}$. 
\end{defn}

\begin{defn}
\cite[Definition 3.1]{DKM09} Let $X$ be a $d$-complex and $k\leq d$.
A $k$-dimensional \emph{simplicial spanning tree} ($k$-SST for short)
of $X$ is a $k$-dimensional subcomplex $Y\subset X$ such that $Y^{k-1}=X^{k-1}$,
$H_{k}\left(Y;\mathbb{Z}\right)=0$ and $\left|H_{k-1}\left(Y;\mathbb{Z}\right)\right|<\infty$,
where $H_{l}\left(Y;\mathbb{Z}\right)$ are the homology groups with
coefficients in $\mathbb{Z}$ (see \cite[Section 2.1]{Ha02} for the
definition). \end{defn}
\begin{lem}
\label{lem:conditions_for_invertability} Let $X$ be a finite $d$-complex. 
\begin{enumerate}
\item When $d=1$, $\Delta_{X\backslash A}^{+}$ is invertible if and only
if $A$ contains a vertex in each of the $0$-components of $X$. 
\item If $A$ is an exhaustive set for the complex $X$, then $\Delta_{X\backslash A}^{+}$
is invertible. 
\item If $A$ is a deformation retract of $X$ and $\Delta_{X\backslash A}^{+}$
is invertible then $H_{d}\left(X\right)=0$.
\item If there exists $\varrho\in X^{d-2}$ such that $\mathrm{cf}\left(\varrho\right)\subset X^{d-1}\backslash A$,
then $\Delta_{X\backslash A}^{+}$ is not invertible. 
\item If $\left|A\right|=\left|X^{d}\right|$ then $\Delta_{X\backslash A}^{+}$
is invertible if and only if $X$ is a $d$-SST of $X$ and $X^{d-2}\cup\left(X\backslash A\right)$
is a $\left(d-1\right)$-SST of $X$. 
\end{enumerate}
\end{lem}
\begin{proof}
$ $
\begin{enumerate}
\item For every $f\in\Omega^{0}\left(X\right)$
\[
\delta_{1}^{X\backslash A}f=\sum_{x\sim y}\left(f\left(x\right)\chi_{x\notin A}-f\left(y\right)\chi_{y\notin A}\right).
\]
Thus $\delta_{1}^{X\backslash A}f=0$ implies that $f$ is constant
on every connected component and is zero on every component containing
a vertex in $A$. 
\item Assume that $A$ is exhaustive with an exhausting sequence $\left(A_{n}\right)_{0\leq n\leq N}$
and that $f\in\ker\Delta_{X\backslash A}^{+}=\ker\delta_{d}^{X\backslash A}$.
For every $\sigma\in A_{1}\backslash A_{0}$ one can find $v\triangleleft\sigma$
such that all $\left(d-1\right)$-faces of $v\sigma$ except for $\sigma$
itself are in $A_{0}$ and therefore 
\[
0=\delta_{d}^{X\backslash A}f\left(v\sigma\right)=\sum_{i=0}^{d}f\left(\left(v\sigma\right)\backslash\left(v\sigma\right)_{i}\right)\cdot\chi_{\left(v\sigma\right)\backslash\left(v\sigma\right)_{i}\notin A_{\pm}}=f\left(\sigma\right).
\]
Consequently $f|_{\left(A_{1}\backslash A_{0}\right)_{\pm}}\equiv0$.
One can now proceed by induction to show that $f_{\left(A_{n}\backslash A_{0}\right)_{\pm}}\equiv0$
for every $1\leq n\leq N$. Since the case $n=N$ implies $f|_{\left(X^{d-1}\backslash A\right)_{\pm}}\equiv0$
the kernel of $\delta_{d}^{X\backslash A}$ is trivial and the result
follows. 
\item Since $X^{d-2}\cup A$ is a $\left(d-1\right)$-complex, $H_{d}\left(X^{d-2}\cup A\right)=0$.
In addition by Claim \ref{claim:trivial_invertability_equivalent_conditions}
and the assumption that $\Delta_{X\backslash A}^{+}$ is invertible
we have $H_{d}\left(X,A\right)=0$. The result now follows since whenever
$A$ is a deformation retract of $X$ the sequence 
\[
0\to H_{d}\left(X^{d-2}\cup A\right)\rightarrow H_{d}\left(X\right)\rightarrow H_{d}\left(X,A\right)\rightarrow H_{d-1}\left(X^{d-2}\cup A\right)\rightarrow\ldots
\]
is exact (see \cite[Theorem 2.13]{Ha02}). 
\item If $\varrho\in X^{d-2}$ and $\mathrm{cf}\left(\varrho\right)\subset X^{d-1}\backslash A$
then the support of the form $\tilde{f}=\delta_{d-1}\ind_{\varrho}\neq0$
is contained in $X^{d-1}\backslash A$ (thus making it the extension
of $f=\tilde{f}|_{X^{d-1}\backslash A}$). Since $\delta_{d}\tilde{f}=\delta_{d}\delta_{d-1}\ind_{\varrho}=0$,
this implies by Claim \ref{claim:trivial_invertability_equivalent_conditions}
that $\Delta_{X\backslash A}^{+}$ is not invertible. 
\item This is the content of \cite[Proposition 4.1]{DKM09}, see also \cite{Ka83}
for a discussion on the complete skeleton case.
\end{enumerate}
\end{proof}

\section{Lower simplicial branching random walk \label{sec:Lower---branching_random_walk}}

Let $X$ be a $d$-complex such that $M=\sup_{\sigma\in X^{d-1}}\deg\left(\sigma\right)<\infty$.
In \cite{MS13}, Mukherjee and Steenbergen defined a version of the
random walk on simplicial complexes correlated with the lower Laplacian
instead of the upper one. This is done by defining a new neighboring
relation, which we call the adjacency relation, that uses faces instead
of cofaces, see Definition \ref{def:adjacency_relation}. The lower
random walk (which is called by the authors the Dirichlet random walk)
is then defined as follows:
\begin{defn}[{\cite[Definition 3.1]{MS13}}]
 The $p$-lazy, $d$-lower random walk is a Markov chain $\left(Z_{n}\right)_{n\geq0}$
on $X_{\pm}^{d}\uplus\left\{ \Theta\right\} $ (where $\Theta$ is
an additional absorbing state) with transition probabilities 
\[
\mathrm{Prob}\left(Z_{n}=\sigma'\Big|Z_{n}=\sigma\right)=\begin{cases}
p & \quad,\,\sigma'=\sigma,\,\,\,\sigma,\sigma'\neq\Theta\\
\frac{1-p}{\left(M-1\right)\left(d+1\right)} & \quad,\,\sigma'\underset{\shortdownarrow}{\sim}\sigma,\,\,\,\sigma,\sigma'\neq\Theta\\
\left(1-p\right)\left[1-\frac{1}{\left(M-1\right)\left(d+1\right)}\cdot\sum_{\tau\in\mathrm{face}\left(\sigma\right)}\deg\left(\tau\right)\right] & \quad,\,\sigma'=\Theta,\,\,\,\sigma\neq\Theta\\
1 & \quad,\,\sigma',\sigma=\Theta\\
0 & \quad,\,\mbox{otherwise}
\end{cases}.
\]

\end{defn}
As for the upper walk define the heat kernel $\mathbf{p}_{n}^{\shortdownarrow}\left(\sigma,\sigma'\right)=\mathrm{Prob}\left(Z_{n}=\sigma'\bigg|Z_{0}=\sigma\right)$,
the lower expectation process $\mathcal{E}_{n}^{\shortdownarrow}\left(\sigma,\sigma'\right)=\mathbf{p}_{n}^{\shortdownarrow}\left(\sigma,\sigma'\right)-\mathbf{p}_{n}^{\shortdownarrow}\left(\sigma,\overline{\sigma'}\right)$
and its normalized version 
\[
\widetilde{\mathcal{E}}_{n}^{\shortdownarrow}\left(\sigma,\sigma'\right)=\left(\frac{M-1}{p\left(M-2\right)+1}\right)^{n}\mathcal{E}_{n}^{\shortdownarrow}\left(\sigma,\sigma'\right).
\]

The following proposition summarizes some of the results proved in
\cite{MS13} regarding the connection between the $d$-lower random
walk and the $d$-homology of the complex: 
\begin{prop}[{\cite[Proposition 1.1]{MS13}}]
Let $X$ be a finite $d$-complex such that $M=\sup_{\sigma\in X^{d-1}}\deg\left(\sigma\right)<\infty$
and $\tilde{\Delta}_{d}^{-}$ the $d$-lower Laplacian%
\footnote{This specific lower Laplacian is related to the weight function $w\equiv1$.%
} given by $\tilde{\Delta}_{d}^{-}f\left(\sigma\right)=\left(d+1\right)f\left(\sigma\right)-\sum_{\sigma'\underset{\downarrow}{\sim}\sigma}f\left(\sigma'\right)$.
\begin{enumerate}
\item The time evolution of $\widetilde{\mathcal{E}}_{n}^{\shortdownarrow}\left(\sigma,\sigma'\right)$
is given by $\widetilde{\mathcal{E}}_{n}^{\shortdownarrow}\left(\cdot,\sigma'\right)=\left(B_{p}\widetilde{\mathcal{E}}_{n-1}^{\shortdownarrow}\right)\left(\cdot,\sigma'\right)$
where $B_{p}$ acts on the first coordinate and is given by 
\[
B_{p}=\frac{p\left(M-2\right)+1}{M-1}I-\frac{1-p}{\left(M-1\right)\left(d+1\right)}\tilde{\Delta}_{d}^{-}.
\]

\item If $\frac{M-2}{3M-4}<p<1$, then $\widetilde{\mathcal{E}}_{\infty}^{\shortdownarrow}=\lim_{n\to\infty}\widetilde{\mathcal{E}}_{n}^{\shortdownarrow}$
always exists. 
\item If $\frac{M-2}{3M-4}<p<1$, then $\left\{ \widetilde{\mathcal{E}}_{\infty}^{\shortdownarrow}\left(\sigma,\cdot\right)\right\} _{\sigma\in X_{\pm}^{d-1}}\subset B^{d-1}$
if and only if $H_{d}\left(X\right)=0$. 
\item If furthermore $p\geq\frac{1}{2}$ then 
\[
\mathrm{dist}\left(\widetilde{\mathcal{E}}_{n}^{\downarrow},\widetilde{\mathcal{E}}_{\infty}^{\downarrow}\right)=O\left(\left(1-\frac{1-p}{\left(p\left(M-2\right)+1\right)\left(d+1\right)}\tilde{\lambda}^{\shortdownarrow}\left(X\right)\right)^{n}\right),
\]
where $\tilde{\lambda}^{\shortdownarrow}\left(X\right)=\min\left(\mathrm{Spec}\left(\tilde{\Delta}_{d}^{-}\Big|_{\left(B^{d+1}\right)^{\bot}}\right)\right)=\min\left(\mathrm{Spec}\left(\tilde{\Delta}_{d}^{-}\right)\right).$
\end{enumerate}
\end{prop}
\begin{rem}
One can also use the lower expectation process in order to find the
dimension of $H_{d}\left(X\right)$ as in Theorem \ref{thm:(d-1)-walk_vs_homology}
and Theorem \ref{thm:Finite_complexes_BRW_vs_homology}. 
\end{rem}
As in the case of the upper Laplacian we define a new stochastic process,
called the \emph{lower simplicial branching random walk}, LSBRW for
short, which is connected to the spectrum of the lower Laplacian in
a similar way as the $d$-lower random walk. The effective process
generated from the LSBRW has the property of being a sequence of random
forms in $\Omega^{d}\left(X\right)$ already in the process level.
In addition it doesn't require normalization and there is no need
for the additional absorbing state.

\medskip{}

The $p$-lazy lower simplicial branching random walk on $X$ is a
time-homogeneous Markov chain $\left(N_{n}^{\downarrow}\left(\cdot\right)\right)_{n\geq0}$
with state space $\mathbb{N}^{X_{\pm}^{d}}$ which counts the number
of particles at time $n$ on any of the oriented $d$-cells, that
is:
\begin{itemize}
\item $N_{n}^{\shortdownarrow}$ is a random function from $X_{\pm}^{d}$
to $\mathbb{N}$.
\item The process is Markovian, i.e., $\mathrm{Prob}\left(N_{n}^{\shortdownarrow}\in A\Big|N_{1}^{\shortdownarrow},\ldots,N_{n-1}^{\shortdownarrow}\right)=\mathrm{Prob}\left(N_{n}^{\shortdownarrow}\in A\Big|N_{n-1}^{\shortdownarrow}\right)$
and time homogeneous, namely $\mathrm{Prob}\left(N_{n}^{\shortdownarrow}=g\Big|N_{n-1}^{\shortdownarrow}=f\right)$
doesn't depend on $n$.
\item $N_{n}^{\shortdownarrow}\left(\sigma\right)$ is the random number
of particles in $\tau$ at time $n$ for every $\tau\in X_{\pm}^{d}$
and $n\geq0$. 
\end{itemize}
One step evolution of the process (its transition kernel) is defined
as follows: Given a configuration of particles on $X_{\pm}^{d}$ all
the particles evolve simultaneously and independently. If a particle
is positioned in $\tau$, then it stays put with probability $p$,
and with probability $1-p$ chooses one of the faces of $\tau$ uniformly
at random and splits into new particles which are now positioned on
the $d$-cells adjacent to $\tau$ whose intersection with $\tau$
is the chosen face (one particle on each such $d$-cell). Note that
one step of the process is comprised of the evolution of all existing
particles. An illustration of one step of the lower simplicial branching
random walk on a triangle complex can be found in Figure \ref{fig:One_step_of_the_lower_branching_process}.

\begin{figure}[h]
\centering{}\includegraphics[scale=1.3]{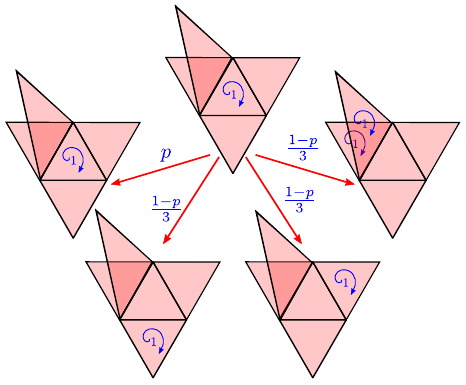}\protect\caption{One step of the LSBRW on a triangle complex.\label{fig:One_step_of_the_lower_branching_process}}
\end{figure}

The effective LSBRW is now defined by 
\[
D_{n}^{\shortdownarrow}\left(\sigma\right)=N_{n}^{\shortdownarrow}\left(\sigma\right)-N_{n}^{\shortdownarrow}\left(\overline{\sigma}\right),
\]
and its heat kernel is 
\[
\mathscr{E}_{n}^{\shortdownarrow}\left(\sigma,\sigma'\right)=E^{\sigma}\left[D_{n}^{\shortdownarrow}\left(\sigma'\right)\right].
\]

As before, the notation $\mathscr{E}_{n}^{\shortdownarrow,p}\left(\sigma,\sigma'\right)$
is used to stress the dependence on $p$.

\begin{rem}
Similarly to the case of SBRW, one can consider several variants of
the process. One can also define the model on $k$-cells in a $d$-complex
for every $1\leq k\leq d$. This however is equivalent to studying
the process on $X^{k}$ instead of $X$ and thus falls back to the
above setting. 
\end{rem}
A similar argument to the one in Lemma \ref{lem:comparison_of_models}(1)
shows that for $n\geq1$ 
\begin{align*}
\mathscr{E}_{n}^{\shortdownarrow}\left(\sigma,\sigma'\right) & =p\mathscr{E}_{n-1}^{\shortdownarrow}\left(\sigma,\sigma'\right)+\frac{1-p}{d+1}\sum_{\sigma''\underset{\shortdownarrow}{\sim}\sigma}\mathscr{E}_{n-1}^{\shortdownarrow}\left(\sigma'',\sigma'\right)=\left(I-\left(1-p\right)\Delta_{d}^{-}\right)\mathscr{E}_{n-1}^{\shortdownarrow}\left(\sigma,\sigma'\right),
\end{align*}
where in the last equality the operator acts on the first coordinate
and $\Delta_{d}^{-}$ is the lower Laplacian associated with the weight
function $w_{\downarrow}$ defined in Example \ref{exa:weight_functions}(3).
This in turn implies by a similar argument to the one in Lemma \ref{lem:comparison_of_models}(2)
that 
\[
\mathscr{E}_{n}^{\shortdownarrow,p}\left(\sigma,\sigma'\right)=\widetilde{\mathcal{E}}_{n}^{\shortdownarrow,p'}\left(\sigma,\sigma'\right),
\]
where $p'=\frac{p}{\left(1-p\right)\left(M-2\right)+1}$.

It is now possible to generalize the results proved for ESBRW to its
lower analogue by repeating the arguments in previous Sections. 

\appendix

\section{Appendix }
\begin{proof}[Proof of Claim \ref{claim:k-good_1}]
 For every $f\in\Omega_{L^{2}}^{k}$\emph{
\begin{align*}
\left\Vert \delta_{k}f\right\Vert ^{2} & =\sum_{\tau\in X^{k}}w\left(\tau\right)\left|\delta_{k}f\left(\tau\right)\right|^{2}\leq\sum_{\tau\in X^{k}}w\left(\tau\right)\cdot\binom{k}{2}\sum_{i=0}^{k}\left|f\left(\tau\backslash\tau_{i}\right)\right|^{2}\\
 & =\sum_{\sigma\in X^{k-1}}\binom{k}{2}\cdot\left(\sum_{\tau\in\mathrm{cf}\left(\sigma\right)}w\left(\tau\right)\right)\left|f\left(\sigma\right)\right|^{2}\leq\binom{k}{2}\cdot\left(\sup_{\sigma\in X^{k-1}}\frac{1}{w\left(\sigma\right)}\sum_{\tau\in\mathrm{cf}\left(\sigma\right)}w\left(\tau\right)\right)\left\Vert f\right\Vert ^{2},
\end{align*}
which shows that (\ref{eq:good_weight_functions}) implies that $\delta_{k}$
is bounded. As for the other direction for any $\sigma\in X_{\pm}^{k-1}$
the function $\ind_{\sigma}$ satisfies $\left\Vert \frac{1}{\sqrt{w\left(\sigma\right)}}\ind_{\sigma}\right\Vert ^{2}=1$
and therefore 
\begin{align*}
\left\Vert \delta_{k}\ind_{\sigma}\right\Vert ^{2} & =\sum_{\tau\in X^{k}}w\left(\tau\right)\left|\delta_{k}\left(\frac{1}{\sqrt{w\left(\sigma\right)}}\ind_{\sigma}\left(\tau\right)\right)\right|^{2}=\frac{1}{w\left(\sigma\right)}\sum_{\tau\in\mathrm{cf}\left(\sigma\right)}w\left(\tau\right).
\end{align*}
Thus, whenever $\sup_{\sigma\in X^{k-1}}\frac{1}{w\left(\sigma\right)}\sum_{\tau\in\mathrm{cf}\left(\sigma\right)}w\left(\tau\right)=\infty$
the operator $\delta_{k}$ is not bounded. }
\end{proof}

\begin{proof}[Proof of Claim \ref{claim:k-good_2}]
 If $\deg\left(\sigma\right)<\infty$ for every $\sigma$, then $\partial_{k}g\left(\sigma\right)$
is a finite sum for every $\sigma\in X_{\pm}^{d-1}$ and is thus well
defined. If \ref{eq:good_weight_functions} holds, then for every
$g\in\Omega_{L^{2}}^{k}$ 
\begin{align*}
\left\Vert \partial_{k}g\right\Vert ^{2} & =\sum_{\sigma\in X^{k-1}}\frac{1}{w\left(\sigma\right)}\left|\sum_{v\vartriangleleft\sigma}w\left(v\sigma\right)g\left(v\sigma\right)\right|^{2}\leq\sum_{\sigma\in X^{k-1}}\frac{1}{w\left(\sigma\right)}\left(\sum_{\tau\in\mathrm{cf}\left(\sigma\right)}w\left(\tau\right)\right)\left(\sum_{v\vartriangleleft\sigma}w\left(v\sigma\right)\left|g\left(v\sigma\right)\right|^{2}\right)\\
 & \leq\sup_{\sigma\in X^{k-1}}\left(\frac{1}{w\left(\sigma\right)}\sum_{\tau\in\mathrm{cf}\left(\sigma\right)}w\left(\tau\right)\right)\cdot\left(\sum_{\sigma\in X^{k-1}}\sum_{\tau\in\mathrm{cf}\left(\sigma\right)}w\left(\tau\right)\left|g\left(\tau\right)\right|^{2}\right)\\
 & =k\cdot\sup_{\sigma\in X^{k-1}}\left(\frac{1}{w\left(\sigma\right)}\sum_{\tau\in\mathrm{cf}\left(\sigma\right)}w\left(\tau\right)\right)\cdot\left\Vert g\right\Vert ^{2}.
\end{align*}

\end{proof}

\begin{proof}[Proof of Claim \ref{claim:trivial_invertability_equivalent_conditions}]
 The equivalence of the first three conditions follows from: 
\begin{align}
\left\langle \Delta_{X\backslash A}^{+}f,f\right\rangle _{X\backslash A} & =\left\langle \delta_{d}^{X\backslash A}f,\delta_{d}^{X\backslash A}f\right\rangle =\left\langle \delta_{d}\tilde{f},\delta_{d}\tilde{f}\right\rangle =\sum_{\tau\in X^{d}}\left(\sum_{i=0}^{d}f\left(\tau\backslash\tau_{i}\right)\cdot\chi_{\tau\backslash\tau_{i}\notin A_{\pm}}\right)^{2},\label{eq:dirichlet_problem_condition_for_trivial_kernel}
\end{align}
where $\left\langle \cdot,\cdot\right\rangle _{X\backslash A}$ is
the inner product $\left\langle \cdot,\cdot\right\rangle $ restricted
to $X^{d-1}\backslash A$. As for the last claim, the equivalence
between $\ker\delta_{d}^{X\backslash A}=0$ and $H_{d}\left(X,A\right)=0$
follows directly from the definition of the relative homology, see
\cite[Section 2.1]{Ha02}. 
\end{proof}
\bibliographystyle{alpha}
\bibliography{citations}

\medskip{}

\lyxaddress{Department Mathematik\\
ETH Zürich \\
CH-8092 Zürich \\
Switzerland\\
E-mail: ron.rosenthal@math.ethz.ch}
\end{document}